\documentclass[12pt]{article}
\usepackage[utf8]{inputenc}
\usepackage{chapterbib}
\usepackage[sectionbib,numbers]{natbib}
\usepackage{eurosym}
\usepackage{amstext,amsthm}
\usepackage{amsmath}
\usepackage{amssymb,mathtools}
\usepackage[nointegrals]{wasysym} 
\usepackage{mathrsfs}
\usepackage{hyperref}
\usepackage{graphicx}
\usepackage{bbm}

\newtheorem{proposition}{Proposition}[section]

\newtheorem{theorem}{Theorem}
\newtheorem{lemma}[proposition]{Lemma}
\theoremstyle{definition}

\newtheorem{remark}[proposition]{Remark}
\newtheorem{assumption}[proposition]{Assumption}
\newtheorem{example}[proposition]{Example}

\usepackage{tikz}
\numberwithin{equation}{section}

\usetikzlibrary{automata, arrows}

\usepackage[normalem]{ulem}
\usepackage{color}

\usepackage{times}

\usepackage{chemarrow}					
\usepackage{faktor}

\begin{document}
\title{\LARGE A unified framework for limit results in Chemical Reaction Networks on multiple time-scales}

\author{\sc Timo Enger, Peter Pfaffelhuber}

\date{\today}

\maketitle

\begin{abstract}
\noindent
If $(X^N)_{N=1,2,...}$ is a sequence of Markov processes which solve the martingale problems for some operators $(G^N)_{N=1,2,...}$, it is a classical task to derive a limit result as $N\to\infty$, in particular a weak process limit with limiting operator $G$. For slow-fast systems $X^N = (V^N, Z^N)$ where $V^N$ is slow and $Z^N$ is fast, $G^N$ consists of two (or more) terms, and we are interested in weak convergence of $V^N$ to some Markov process $V$. In this case, for some $f\in \mathcal D(G)$, the domain of $G$, depending only on $v$, the limit $Gf$ can sometimes be derived by using some $g_N\to 0$ (depending on $v$ and $z$), and study convergence of $G^N (f + g_N) \to Gf$. We develop this method further in order to obtain functional Laws of Large Numbers (LLNs) and Central Limit Theorems (CLTs). We then apply our general result to various examples from Chemical Reaction Network theory. We show that we can rederive most limits previously obtained, but also provide new results in the case when the fast-subsystem is first order. In particular, we allow that fast species to be consumed faster than they are produced, and we derive a CLT for Hill dynamics with coefficient~2.
\end{abstract}

\section{Introduction}
The motivation for our work and the resulting limiting results comes from recent publications on multi-scale (Markovian) Chemical Reaction Networks (CRNs). Here, multi-scale refers to the fact that  reaction rates differ by orders of magnitude; see e.g.\ \cite{AndersonKurtz2015}. In contrast, limit results for single-scale systems, where all species are in the same (high) abundance, have basically been solved by early work of Tom Kurtz (\cite{Kurtz1970, Kurtz1971}), providing both, functional Laws of Large Numbers (LLNs), with the solution of a system of Ordinary Differential Equations (ODEs) as limit, and functional Central Limit Theorems (CLTs), leading to the solution of a system of It\^o Stochastic Differential Equations (SDEs) as limit. For multi-scale reaction networks, frequently the distinction between deterministic (if all species abundances are large) and stochastic networks (if all species abundances are small) is made. In deterministic models, the quasi-steady-state assumption is most important; see e.g.\ \cite{Heineken2019}. In early papers on  stochastic models \cite{HaseltineRawlings2002, RaoArkin2003}, the focus was on simulation techniques of slow-fast systems. Rather than distinguishing strictly between deterministic and stochastic models, we will use the notions introduced in \cite{BallKurtzPopovicRempala2006}, where some species are in large and some in small abundance, and reaction rates still differ by orders of magnitude. Such multi-scale systems still pose interesting fundamental questions today. LLNs in these systems have been derived using techniques involving stochastic averaging \cite{Kurtz1992} and approximations of Poisson processes by Brownian motion; see \cite{BallKurtzPopovicRempala2006, KangKurtz2013}. The  case of LLNs involving fast intermediate species is treated  in \cite{CappellettiWiuf2016}, \cite{SaezWiufFeliu2017} using probabilistic arguments, i.e.\ exit times and probabilities of the fast sub-system. Here, fast species in low copy number, are produced on a slower time-scale than they are consumed, and are called intermediate. This concept was recently generalized to non-interacting species in \cite{HoesslyWiufXia2021}. 
Concerning functional CLTs in such multi-scale systems, \cite{KangKurtzPopovic2014} use the technique of stochastic averaging, solutions of Poisson equations and the martingale central limit theorem for proving limit results on CRNs on two or three time-scales. Their results have been used in various contexts, e.g.\ for a CRN involving the heat shock protein of {\em E.\ coli} \cite{Kang2012}. In this theory, a central role is played by the averaging condition (their Condition 2.4), which assumes a unique equilibrium of fast species when the slow species are fixed. A related condition is the balance condition from \cite{KangKurtz2013}, which assumes that all species -- in particular fast species -- are produced and consumed on the same time-scale. Both, the averaging condition and the species balance condition, are violated by the intermediate species from \cite{CappellettiWiuf2016}, since they  are consumed faster than they are produced and hence no limit averaging distribution exists. 

In our setup, we will assume that  species in low abundance are fast, and species in high abundance are slow. For such systems, our goal is threefold: We want to (i) provide a general framework which provides the possibility to derive limit results for all examples in the above mentioned papers, in particular (ii) provide a method to derive functional CLTs for the case of intermediate species, and (iii) obtain CLTs in cases not treated so far. In order to achieve this, we give in Theorem~\ref{T1} a general result which requires only an application of a well-known limit result for Markov processes (Theorem~2.1 in \cite{Kurtz1992}). Applying this general result to specific cases of CRNs leads to a {\em recipe} how to obtain LLNs and CLTs, described in Section~\ref{ss:32}. We provide concrete computations on previous examples and how our theory applies in these contexts; see Section~\ref{S:examples}. In particular, we will treat Michaelis-Menten kinetics (one of the main examples in \cite{BallKurtzPopovicRempala2006, KangKurtzPopovic2014}) in Section~\ref{ss:MM}, but treat also the main example from \cite{CappellettiWiuf2016} in Section~\ref{ss:CWmain}. We add to the list of examples an extension of Michaelis-Menten kinetics, which is known in the literature as dynamics with Hill coefficient~2; see Examples~\ref{ex:simHill1}--\ref{ex:simHill5} and  Section~\ref{ss:Hill}.

\section{Main Results}
Before we provide with Theorem~\ref{T1} our main general result, we give some heuristic arguments and an example. Recall that a stochastic process $V$ with Polish state space $E$ solves the $(G, \mathcal D(G))$ martingale problem for some linear $G: \mathcal D(G) \subseteq \mathcal C_b(E) \to \mathcal C_b(E)$ if 
\begin{align*}
  \Big(f(V_t) - \int_0^t G f(V_s) ds\Big)_{t\geq 0}
\end{align*}
is a martingale for all $f\in\mathcal D(G)$. Assume that there is a unique (in law) solution, i.e.\ the martingale problem is well-posed and $V$ is unique and Markovian, if $\mathcal D(G)$ is large enough. 

Our goal is to obtain convergence of a sequence $V^N$ to some $V$ as $N\to\infty$. Here, $V^N$ is one coordinate in the Markov process $(V^N, Z^N)$ with state space $E\times F$ (for $E, F$ Polish), which is the solution of a martingale problem with operators of the form
\begin{align*}
  G^N & = G^N_0 + N^{1/2} G^N_1 + N G^N_2.
\end{align*}
We will call $V^N$ slow and $Z^N$ fast, where {\em slow} and {\em fast} means that $G^N_2f = 0$ if $f \in \mathcal D(G^N) \supseteq \mathcal D(G)$ only depends on $v$ (i.e.\ generator term of the fastest dynamics, $G_2^N$ describes only the dynamics of $Z^N$). For this convergence, we assume that for every $f \in \mathcal D(G)$ (in particular, $f$ only depends on $v$), there are $g_N,h_N \in \mathcal D(G^N)$ such that (recall $G_2^Nf = 0$)
\begin{align}\label{eq:12}
    Gf(v) \approx G_0^N f(v) + (G_1^N g_N + G_2^N h_N)(v,z) + N^{1/2} (G_1^N f + G_2^Ng_N)(v,z)
\end{align}
in the sense that for all $T>0$
\begin{align*}
& \mathbb E\Big[\int_0^T |\varepsilon_f^N(V_t^N, Z_t^N)|dt\Big]  \xrightarrow{N\to\infty} 0 \text{ for}
\\ \notag
& \varepsilon_f^N(v,z):= G_0^Nf(v) + (G^N_1g_N+G_2^N h_N)(v,z) + N^{1/2}(G^N_1f+G^N_2 g_N)(v,z) -Gf(v).
\end{align*}
Then, provided tightness of $V^N$ and some boundedness restrictions of $g_N, h_N, G_0^N g_N, G_0^Nh_N$ and $G_1^Nh_N$, and $V'$ is an accumulation point of $V^N$, it is reasonable to conclude that
\begin{equation} \notag 
\begin{aligned}
    f(V^N_t) & + (N^{-1/2}g_N + N^{-1}h_N)(V_t^N, Z_t^N) - \int_0^t G^N(f + N^{-1/2}g_N + N^{-1}h_N)(V_s^N, Z_s^N)ds
    \\ & \approx f(V_t^N) -  \int_0^t (G_0^Nf + G^N_1g_N+G_2^N h_N + N^{1/2}(G^N_1f+G^N_2 g_N))(V_s^N, Z_s^N)ds \\ & \approx f(V_t') - \int_0^t Gf(V_s') ds 
\end{aligned}
\end{equation}
is a martingale. If the martingale problem for $(G, \mathcal D(G))$ is well-posed with unique solution $V$, this then shows convergence $V^N\xRightarrow{N\to\infty} V$, provided the initial conditions converge. 

\begin{example}[Convergence to Brownian Motion]
  For applying the result, let us have a look at \eqref{eq:12}. Since the term $N^{1/2}(G_1^Nf + G_2^Ng_N)$ is the leading term on the right-hand-side, we first look for $g_N$ such that $G_1^Nf + G_2^Ng_N = 0$, and then for $h_N$ such that $G_1^N g_N + G_2^N h_N$ only depends on $v$. 

  Let us consider a simple example outside of Chemical Reaction Networks, which shows that the above method can work: Let $X^N = (V^N,Z^N)$ be the Markov process with state space $\mathbb R \times \mathbb R$ and generator, for some probability measure $\pi$ with mean zero and variance $\tfrac 12$, writing $f'$ for the derivative with respect to the first coordinate,
  \begin{align*}
      G^Nf(v,z) = N^{1/2} \underbrace{z f'(v,z)}_{=G_1f(v,z)} + N \underbrace{\int (f(v,z') - f(v,z)) \pi(dz')}_{=G_2f(v,z)},
  \end{align*}
  i.e.\ $V^N$ is a continuous random walk, which changes slope at rate $N$. Clearly, $V^N$ converges to a Brownian motion. In order to see this, for $f\in \mathcal C^\infty_b(\mathbb R)$, we choose 
  $$g_N(v,z) := g(v,z) := z f'(v), \text{ which implies }G_1f + G_2 g_N = 0$$ 
  (since $\pi$ has mean zero). Next, note that $G_1 g(v,z) = z^2 f''(v)$, so we choose $$h_N(v,z) := h(v,z)=z^2 f''(v), \text{ which implies }(G_1g + G_2 h)(v,z) = Gf(v) := \tfrac 12 f''(v),$$ leading to $\varepsilon_N^f=0$ in \eqref{eq:13}. Hence, provided tightness is shown, the desired convergence follows.
\end{example}

\begin{theorem}[Convergence result]
\label{T1}
  Let $E,F$ be Polish and $(X^N = (V^N, Z^N))_{N=1,2,...}$ be a sequence of Markov processes with state space $E\times F$, and with generators of the form
  \begin{align}\label{eq:11}
    G^N & = G^N_0 + N^{1/2} G^N_1 + N G^N_2,
  \end{align}
  such that $(V^N)_{N=1,2,...}$ satisfies the compact containment condition, i.e.\ for all $\varepsilon>0$ and $T>0$, there is $K\subseteq E$ compact with
  $$ \inf_N \mathbb P( V^N_t \in K, 0\leq t \leq T) > 1 - \varepsilon,$$
  and let $G:  \mathcal D(G) \subseteq \mathcal C_b(E) \cap \mathcal D(G^N) \to \mathcal C_b(E)$. 
  Assume that for every $f\in \mathcal D(G)$ and $N=1,2,...$, there is $g_N, h_N\in \mathcal D(G^N)$, such that for all $T>0$
  \begin{align}\label{eq:T1a}
    N^{-1/2}\mathbb E[\sup_{t\leq T} |g_N(V_t^N, Z_t^N)|] + N^{-1} \mathbb E[\sup_{t\leq T} |h_N(V_t^N, Z_t^N)|] &\xrightarrow{N\to\infty} 0,    \\ \label{eq:T1b}
    N^{-1/2} \mathbb E\Big[ \int_0^T \big|(G_0^N g_N + G_1^N h_N)(V_t^N, Z_t^N)\big| dt\Big] + N^{-1} \mathbb E\Big[ \int_0^T |G_0^N h_N(V_t^N, Z_t^N)|dt\Big]  & \xrightarrow{N\to\infty} 0
    \intertext{and}
    \label{eq:T1d}
    \mathbb E\Big[ \int_0^T |\varepsilon_f^N(V_t^N, Z_t^N)| dt\Big] & \xrightarrow{N\to\infty} 0
  \end{align}
  for 
  \begin{align}\label{eq:defeps}
    \varepsilon_f^N(v,z) = G_0^Nf(v) + G^N_1g_N(v,z)+G_2^N h_N(v,z) + N^{1/2}(G^N_1f(v)+G^N_2 g_N(v,z)) - Gf(v). 
  \end{align}
  Then, $(V^N)_{N=1,2,...}$ is tight and for every limit point $V$,   \begin{align}\label{eq:limitMP}
      \Big(f(V_t) - \int_0^t Gf(V_s)ds\Big)_{t\geq 0}
  \end{align}
  is a martingale for each $f\in\mathcal D(G)$.
\end{theorem}

\begin{remark}[Connections to other convergence results]
\begin{enumerate}
    \item 
    Our proof relies on Theorem~2.1 of \cite{Kurtz1992}, as we will see in the below. However, this result is more general since it treats the case where the limiting generator ($A$ in the notation of \cite{Kurtz1992}) arises as a limit of generators which are integrated over the occupation measure of the fast variable (rather than not depending on the fast variable at all). So, it is not surprising that Theorem~\ref{T1} holds, since it treats just a special case where the limiting generator ($A$) does not depend on the fast variable. However, it is maybe surprising that Theorem~\ref{T1} is useful in obtaining limit results in a bunch of examples.
  \item Consider the case where the fast   variable jumps according to
    $$ G_2^Nf(v,z) = \int f(v,y) \pi_N(dy) - f(v,z),$$
    i.e.\ $Z^N$ jumps to new states according to the equilibrium measure $\pi_N$ at rate $N$. We note that this form of $G_2^N$ implies that $f + G_2^N f$ does not depend on $z$. Then, for smooth $f$, we set $g_N = G_1^Nf$ and $h_N := G_1^N G_1^N f$. If the limits
  \begin{align*}
      A_1 f(v) & := \lim_{N\to\infty} N^{1/2} \int G_1^Nf(v,y) \pi_N(dy), \\ A_2 f(v) & := \lim_{N\to\infty} G_0^N f(v,z) + \int G_1^N G_1^N f(v,y) \pi_N(dy)
  \end{align*}
  exist, we then see from Theorem~\ref{T1} that -- provided 
  the compact containment condition of $(V^N)_{N=1,2,...}$ and \eqref{eq:T1a}--\eqref{eq:T1d} hold -- the weak limit of $V^N$ has generator $G:=A_1 + A_2$. This case is actually treated in Corollary 2.5 of \cite{HutzenthalerPfaffelhuberPrintz2021}. Their main result, Theorem 2.3, is more general since it allows for more general $G_2^N$, and $A_2$ from above may arise by an application of stochastic averaging. 
\end{enumerate}
  
\end{remark}

\begin{proof}
  We are going to apply Theorem~2.1 of \cite{Kurtz1992}. In our notation (or their notation), we will have the special case where $Gf$ (or $Af)$ does not depend on the fast variable $z$ (or $y$) for the limiting generator $G$ (or $A$) and $f\in\mathcal D(G)$ (or $f\in\mathcal D(A)$). \\ Since $Gf$ is bounded by assumption, their (2.3) is satisfied and their (2.6) becomes our \eqref{eq:limitMP}. So, it remains to properly define some $\delta_f^N$ (or $\varepsilon_f^N$) in their (2.2), which satisfies their (2.4). Here, we have for $f\in\mathcal D(G)$ the martingale
  \begin{align*}
      f(V_t^N) & + (N^{-1/2} g_N + N^{-1} h_N)(V_t^N, Z_t^N) - \int_0^t G^N (f + N^{-1/2} g_N + N^{-1} h_N)(V_s^N, Z_s^N)ds
      \\ & = f(V_t^N) - \int_0^t Gf(V_s^N) ds + \delta_f^N(t)
  \end{align*}
  with
  \begin{align*}
      \delta_f^N(t) & = (N^{-1/2} g_N + N^{-1} h_N)(V_t^N, Z_t^N) - \int_0^t \varepsilon_f^N(V_s^N, Z_s^N) ds \\ & \qquad - \int_0^t (N^{-1/2} (G_0^N g_N + G_1^N h_N) + N^{-1} G_0^N h_N) (V_s^N, Z_s^N)ds.
  \end{align*}
  By \eqref{eq:T1a}, \eqref{eq:T1b} and \eqref{eq:T1d}, we find that $\mathbb E[\sup_{0\leq t\leq T}|\delta_f^N(t)|] \xrightarrow{N\to\infty} 0$, and we are done.
\end{proof}

\begin{remark}[Two and more than three scales]
\begin{enumerate}
    \item If we have a two-scale system, i.e.\ we have $G_1^N=0$ in \eqref{eq:11}, we can choose $g_N=0$, and have to find $h_N$ satisfying some boundedness restrictions and $\mathbb E\Big[\int_0^T|\varepsilon_f^N(V_t^N, Z_t^N)|dt\Big]  \xrightarrow{N\to\infty} 0$ for
    \begin{align}\label{eq:13}& \varepsilon_f^N(v,z):= (G_0^Nf + G_2^N h_N - Gf)(v,z).
  \end{align}
  We will use this case in functional LLNs in Section~\ref{S:examples}.
  \item For systems of more than three scales, a similar approach works equally well. Let $X^N = (V^N, Z^N)$ be a Markov process with generator
  \begin{align*}
    G^N & = G^N_0 + N \cdot G^N_1 + N^2 \cdot G^N_2 + N^3 \cdot G^N_3
    \end{align*}
  and $G_3^N f = 0$ if $f$ only depends on $v$. For a limiting generator $G$ and $f\in \mathcal D(G)$ (in particular depending only on $v$), we assume there are $g_N, h_N, k_N \in \mathcal D(G^N)$ such that \sloppy $\mathbb E\Big[\int_0^T|\varepsilon_f^N(V_t^N, Z_t^N)|dt\Big]  \xrightarrow{N\to\infty} 0$ for
\begin{align*}
  \varepsilon_f^N(v,z) & = (G_0^Nf + G_1^Ng_N + G_2^Nh_N + G_3^Nk_N  \\ & \qquad+ N(G_1^Nf + G_2^Ng_N + G_3^Nh_N) + N^2(G_2^Nf + G_3^Ng_N) - Gf)(v,z),
\end{align*}
as well as the compact containment condition of $(V^N)$ and some boundedness assumptions for $g_N$, $h_N$, $k_N$, $G_0^Ng_N$, $G_i^Nh_N, i=0,1$, $G_i^Nk_N, i=0,1,2$, convergence of $V^N$ to a solution of the $(G, \mathcal D(G))$ martingale problem can be shown as well.
\end{enumerate}
\end{remark}

\section{Applications to Chemical Reaction Networks (CRNs)}
\label{S:CRNs}
Our goal is to apply Theorem~\ref{T1} to Chemical Reaction Networks. While describing the framework in Section~\ref{ss:CRNs}, we will introduce a system which leads to Hill dynamics with coefficient~2 as a first example. After stating the main assumptions (above all, the fast subsystem is first order) in Section~\ref{ss:32}, we describe in Section~\ref{ss:33} how to derive LLNs and in Section~\ref{ss:34} the corresponding CLTs in this setting. In Section~\ref{S:ccc}, we give some general remarks how to show that the assumptions (e.g.\ the compact containment condition) from Theorem~\ref{T1} hold.

\subsection{CRNs and their rescaling}
\label{ss:CRNs}
Chemical reaction networks can be described as a pair $(\mathcal S, \mathcal R)$, consisting of the (finite) set of species $\mathcal S$ and the (finite) set of reactions $\mathcal R$, where we
write
\begin{align*}
  \sum_{S\in\mathcal S} \nu_{RS} S \xrightarrow{\text{\phantom{xx}}\widetilde\kappa_R\text{\phantom{xx}}}  \sum_{S\in\mathcal S} \nu'_{RS} S, \qquad R \in \mathcal R 
\end{align*}
for some $\nu, \nu' \in \mathbb N_0^{\mathcal R \times \mathcal
  S}$. Using this notation, the Markov dynamics
$X = (X_S)_{S\in\mathcal S}$, $X_S = (X_S(t))_{t\geq 0}$ 
can be described through the jump rates: Using the vector of species consumed and species produced of the reactions, 
\begin{align*}
  \nu_{R\cdot} & := (\nu_{RS})_{S\in\mathcal S}, \qquad \nu_{R\cdot}' := (\nu_{RS}')_{S\in\mathcal S}, \qquad R\in\mathcal R
  \intertext{and the falling factorial}
  n_{[k]} & := n\cdot (n-1) \cdots (n-k+1), \qquad k,n \in \mathbb N_0,
\end{align*}
starting in $x(0)$, the process jumps
\begin{align*}
  \text{from }X=x \text{ to } x + (\nu_{R\cdot}' - \nu_{R\cdot}) \text{ at rate } \lambda_R(x)
\end{align*}
for some $\lambda_R$. In the sequel, we will be using mass action kinetics, i.e.\ we assume that
\begin{align*}
    \lambda_R(x) := \widetilde \kappa_R
  \prod_{S \in \mathcal S} (x_{S})_{[\nu_{RS}]}
\end{align*}
for some $\widetilde\kappa_R \in \mathbb R_+$. However, different rate functions can be used as well. 

~

In order to obtain limit results, we have to introduce a scaling parameter (denoted $N$ in the sequel) and rescaled time of the dynamics $X$. We are not dealing with the completely general situation (see e.g.\ \cite{KangKurtz2013}), but with a set of species which is either in {\em low} or {\em high} abundance. Here, the former are not rescaled, but the latter are rescaled by a factor of $N$. This means that ($\uplus$ meaning a disjoint union)
\begin{align*}
  \mathcal S = \mathcal S_\bullet \uplus \mathcal S_\circ,
\end{align*}
where $S\in\mathcal S_\bullet$ are in high (i.e.\ $\mathcal O(N)$)
and $S\in\mathcal S_\circ$ are in low (i.e.\ $\mathcal O(1)$) abundance. We set $V^N = (V^N_S)_{S\in \mathcal S_\bullet}$ and $Z^N = (Z^N_S)_{S\in \mathcal S_\circ}$ with
\begin{align*}
    V^N_S = N^{-1} X^N_S, \quad S\in\mathcal S_\bullet, \qquad Z^N_S = X^N_S, \quad S\in\mathcal S_\circ.
\end{align*}
For the reactions, we assume that there are $\kappa_R \in \mathbb R_+$
and $\beta_R \in \mathbb R_0, R\in\mathcal R$ such that
\begin{align*}
  \widetilde\kappa_R = \kappa_R N^{\beta_R}.
\end{align*}
Abbreviating $\nu'_{R\mathcal S_\bullet} := (\nu_{RS})_{S\in\mathcal S_\bullet}$, and equivalently for $\nu'_{R\mathcal S_\circ}, \nu_{R\mathcal S_\bullet}, \nu_{R\mathcal S_\circ},$
\begin{align*}
    \zeta_{RS} := \nu'_{RS} - \nu_{RS}, 
\end{align*}
as well as $\zeta_{R\mathcal S_\bullet}, \zeta_{R\mathcal S_\circ}$, the process $(V^N, Z^N)$ then jumps
\begin{equation}
\label{eq:jump2}
  \begin{aligned}
    \text{from }(V^N,Z^N) & = (v,z) \text{ to } (v + N^{-1} \zeta_{R\mathcal S_\bullet}, z +  \zeta_{R\mathcal S_\circ}), \\ & \text{ at rate }      \lambda_R^N(v, z) := N^{\beta_R + \sum_{S\in\mathcal S_\bullet} \nu_{RS}} \kappa_R \cdot z_{[\nu_{R\mathcal S_\circ}]}  \cdot (v_{\mathcal S_\bullet})_{[\nu_{R\mathcal S_\bullet}],N^{-1}}
  \end{aligned}
\end{equation}
where we use 
\begin{align*}
  n_{[k],\varepsilon} = n \cdot (n-\varepsilon) \cdots (n-(k-1)\varepsilon), \qquad & n_{[k]} := n_{[k],1},
  \\
  (n_{\mathcal S_\bullet})_{[k_{\mathcal S_\bullet}], \varepsilon} := \prod_{S\in\mathcal S_\bullet} (n_S)_{[k_S], \varepsilon}, \quad n_{\mathcal S_\bullet} \in \mathbb R^{\mathcal S_\bullet}, \qquad   & (n_{\mathcal S_\circ})_{[k_{\mathcal S_\circ}]} := \prod_{S\in\mathcal S_\circ} (n_S)_{[k_S]}, \quad n_{\mathcal S_\circ} \in \mathbb R^{\mathcal S_\circ}
\end{align*}
and note that (for appropriate $\delta^N$)
\begin{align*}
    v_{[k],N^{-1}} & = v(v-N^{-1}) \cdots (v-(k-1)N^{-1}) = v^k - N^{-1} \binom k2 v^{k-1} + N^{-2} \delta^N(v,k).
\end{align*}
In other words, writing\footnote{We make the convention that $x$ denotes a column vector, and $x^\top$ a row vector.}
\begin{align} \label{eq:kappap}
    \nabla_\bullet f & := (\partial f/\partial v_S)^\top_{S\in\mathcal S_\bullet}, \qquad \kappa_R' := N^{\beta_R - 1 + \sum_{S\in\mathcal S_\bullet} \nu_{RS}} \kappa_R
\end{align}
for all $f \in \mathcal C_c^\infty(\mathbb R^{\mathcal S_\bullet \times \mathcal S_\circ})$, the process
\begin{align*}
    \Big(f(V_t^N, Z_t^N) - \int_0^t G^Nf(V_s^N, Z_s^N) ds\Big)_{t\geq 0}
\end{align*}
is a martingale, where
\begin{equation}\label{eq:CRNgen}
\begin{aligned}
    G^Nf(v,z) & = \sum_{R\in\mathcal R} N\kappa_R' z_{[\nu_{R\mathcal S_\circ}]} v_{[\nu_{R\mathcal S_\bullet}],N^{-1}} (f(v + N^{-1}\zeta_{R\mathcal S_\bullet}, z + \zeta_{R\mathcal S_\circ}) - f(v,z)) \\ & = (G_0^Nf + G_2^Nf)(v,z), \\
    G_0^N f(v,z) & = \sum_{R\in\mathcal R} \kappa_R' z_{[\nu_{R\mathcal S_\circ}]} \big(v^{\nu_{R\mathcal S_\bullet}} \nabla_\bullet f(v, z + \zeta_{R\mathcal S_\circ}) \cdot \zeta_{R\mathcal S_\bullet} \\ & \qquad \qquad \qquad \qquad - \binom{\nu_{R\mathcal S_\bullet}}{2} v^{\nu_{R\mathcal S_\bullet}-1}(f(v, z + \zeta_{R\mathcal S_\circ}) - f(v,z))\big) + \varepsilon_f^N(v,z), \\ G_2^Nf(v,z) & = \sum_{R\in\mathcal R} \kappa_R' z_{[\nu_{R\mathcal S_\circ}]} v^{\nu_{R\mathcal S_\bullet}} (f(v, z + \zeta_{R\mathcal S_\circ}) - f(v,z)), \\ |\varepsilon_f^N(v,z)| & \leq \sum_{R\in\mathcal R} N^{-1} \kappa_R' z_{[\nu_{R\mathcal S_\circ}]}  \big(\tfrac 12 v^{\nu_{R\mathcal S_\bullet}} ||\nabla^2_\bullet f||_\infty + \binom{\nu_{R\mathcal S_\bullet}}{2} v^{\nu_{R\mathcal S_\bullet}-1}||\nabla_\bullet f||_\infty + \delta^N(v,\nu_{R\mathcal S_\bullet})||f||_\infty\big). 
\end{aligned}
\end{equation}
Here $\nabla^2_\bullet f$ is the Hessian of $f$, $\binom{\nu_{R\mathcal S_\bullet}}{2} := \prod_{S\in\mathcal S_\bullet} \binom{\nu_{RS}}{2}$ and $v^{\nu_{R\mathcal S_\bullet}} := \prod_{S\in\mathcal S_\bullet} v^{\nu_{RS}}.$

\begin{example}[The simplified Hill dynamics~1]\label{ex:simHill1}
We illustrate our theory with the following example:
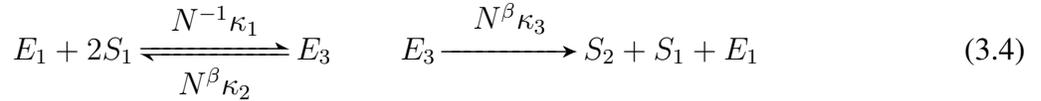
\begin{align}\label{eq:CRNhill2}
  \operatorname{\text{$E_1+2S_1$}} \autorightleftharpoons{\text{$N^{-1}\kappa_{1}$}}{\text{$N^\beta \kappa_{2}$}}   \operatorname{\text{$E_3$}} \qquad  \operatorname{\text{$E_3$}} \autorightarrow{\text{$N^\beta\kappa_3$}}{}    \operatorname{\text{$S_2+S_1+E_1$}}
\end{align}
with $\beta\geq 1$. This CRN arises e.g.\ if a macromolecule ($E_1$) has two binding sites for a substrate ($S_1$), which quickly binds (simultaneously at both binding sites) and forms a complex ($E_3$). If the ligands are released, one of them is turned into $S_2$. We assume that the macromolecule (i.e.\ $E_1, E_3$) is in low abundance $(O(1))$, while substrate and product, $S_1, S_2$, are in high abundance $(O(N))$, so $\mathcal S_\bullet = \{S_1, S_2\}$ and $\mathcal S_\circ = \{E_1, E_3\}$. The dynamics are similar to Michaelis-Menten kinetics below (see Section~\ref{ss:MM}), with the differnce that two substrate molecules are needed in order to form the product. A similar system was introduced by A.~V.~Hill \cite{Hill1910} for binding of oxygen to haemoglobin in order to understand data from osmotic pressure experiments. We will refine this example in Section~\ref{ss:Hill}, where the first and second reaction are broken into two separate reactions for binding of the two substrate molecules. (The intermediate step is the macromolecule bound by a single substrate, denoted $E_2$, which is missing in \eqref{eq:CRNhill2}.) For this simplified dynamics, we derive the LLN and CLT below in Examples~\ref{ex:simHill2}-- \ref{ex:simHill5}.

We let $NV_1^N,NV_2^N$ and $Z_1,Z_3$ be the processes of the particle numbers of $S_1,S_2$ and $E_1,E_3$. Also $M:=Z_1+Z_3$ is a constant as well as $V_1+V_2+N^{-1}(Z_1+3Z_3)$, so we are going to describe the Markov process $(Z^N=Z_1^N,V^N=V_1^N)$. Since $V^N$ is evolving slower than $Z^N$, we call $V^N$ the slow and $Z^N$ the fast variable. If the slow variable is assumed constant, the fast-subsystem (consisting of $E_1, E_3$) is first order; see Assumption~\ref{ass:1} below. However, the slow sub-system is second order. For the generator of $(V^N, Z^N)$ we have 
for $f \in \mathcal C_c^\infty(\mathbb R\times\mathbb R)$, writing $f'$ for the derivative according to the first variable, and (recall from \eqref{eq:kappap}) $\kappa_2' = N^{\beta-1} \kappa_2, \kappa_3' = N^{\beta-1} \kappa_3$,
\begin{align*}
    G^Nf(v,z)&=N\kappa_1v(v-N^{-1})z(f(v-2N^{-1},z-1)-f(v,z))\\
    &\qquad+N\kappa_2'(M-z)(f(v+2N^{-1},z+1)-f(v,z))\\
    &\qquad+N\kappa_3'(M-z)(f(v+N^{-1},z+1)-f(v,z))\\
    &=G_0^Nf(v,z)+NG_2^Nf(v,z),\\
    G_0^Nf(v,z)&=-2\kappa_1v^2zf'(v,z-1)+2\kappa_2'(M-z)f'(v,z+1)\\
    &\qquad +\kappa_3(M-z)f'(v,z+1)-\kappa_1vz(f(v,z-1)-f(v,z)) +\varepsilon_f^N(v,z),\\
    G_2^Nf(v,z)&=\kappa_1v^2z(f(v,z-1)-f(v,z))+\kappa_2'(M-z)(f(v,z+1)-f(v,z))\\
    &\qquad +\kappa_3'(M-z)(f(v,z+1)-f(v,z)),\\
    |\varepsilon_f^N(v,z)| &\leq  N^{-1}\big(\big(\kappa_1v^2z + \kappa_2'(M-z)+\tfrac 12\kappa_3'(M-z)\big) ||f''||_\infty + 2\kappa_1 v z ||f'||_\infty\big).
\end{align*}
\end{example}

\subsection{Main assumptions}
\label{ss:32}
Our main assumption is that the fast sub-system is first order. This means that if all $S\in\mathcal S_\bullet$ are assumed to be constant, the remaining CRN on $\mathcal S_\circ$ is first order, i.e.\ at most a single molecule is consumed and produced in each reaction; see Assumption~\ref{ass:1}.3 below. All our examples 
are of this kind. For $S\in\mathcal S$, we set 
$$\Gamma_S^+ := \{R\in\mathcal R: \nu'_{RS} > \nu_{RS}\}, \qquad \Gamma_S^- := \{R\in\mathcal R: \nu_{RS} > \nu'_{RS}\}$$ as the sets of reactions effectively producing and consuming species $S$, respectively. For a leaner notation, we set
$$ \widetilde \beta_R := \beta_R + \sum_{S\in\mathcal S_\bullet} \nu_{RS}.$$

\begin{assumption}\label{ass:1}
\begin{enumerate}
    \item The reaction rates $\lambda_R^N$ are of the form given in \eqref{eq:jump2}. 
    \item The fast sub-system is first order, i.e.\
    $$ \sum_{S\in\mathcal S_\circ} \nu_{RS} \leq 1, \qquad \sum_{S\in\mathcal S_\circ} \nu_{RS}' \leq 1, \qquad \qquad R\in\mathcal R.$$
    We write $$\mathcal R = \mathcal R_0 \uplus \mathcal R_1 = \mathcal R_0' \uplus \mathcal R_1',$$ where 
    \begin{align*}
          \mathcal R_0 & := \Big\{R\in\mathcal R: \sum_{S\in\mathcal S_\circ} \nu_{RS} = 0 \Big\}, \qquad 
      \mathcal R_1  := \Big\{R\in\mathcal R: \sum_{S\in\mathcal S_\circ} \nu_{RS} = 1 \Big\}, \\      \mathcal R_0' & := \Big\{R\in\mathcal R: \sum_{S\in\mathcal S_\circ} \nu_{RS}' = 0 \Big\}, \qquad  \mathcal R_1'  := \Big\{R\in\mathcal R: \sum_{S\in\mathcal S_\circ} \nu_{RS}' = 1 \Big\}.    
    \end{align*}
    For $R\in\mathcal R_1$ and $S\in\mathcal S_\circ$, we set $S_R = S$ if $\nu_{RS} = 1$, i.e.\ $R$ consumes $S_R$ as single fast species.    For $R\in\mathcal R_1'$ and $S\in\mathcal S_\circ$, we set $S_R' = S$ if $\nu_{RS}' = 1$, i.e.\ $R$ produces $S_R'$ as single fast species.
    \item For $S\in\mathcal S_\circ$, we require that 
    \begin{align}\label{eq:varphi}
        \varphi_S & := \max\Big\{\widetilde\beta_R: R\in\Gamma_S^-\Big\} \geq 1,
    \end{align}
    and for
    \begin{align}\label{hatbeta}
      \widehat\beta_{R} & := \begin{cases} \widetilde \beta_R, & R\in \mathcal R_0, \\ \widetilde\beta_R + 1 - \varphi_{S_R}, & R\in \mathcal R_1.\end{cases}
    \end{align}
    that 
    \begin{equation}\label{eq:prodcun}
    \begin{aligned}
    \psi_S & := \max\{\widehat\beta_{R}: R\in\Gamma_S^+\} = 1,
    \end{aligned}
    \end{equation}
    \item For $S\in\mathcal S_\bullet$, the {\em time-scale constraint} applies, i.e.\
    \begin{equation}\label{eq:prodcun2}
    \begin{aligned}
    \max\Big\{\widehat\beta_R: R\in \Gamma_S^+ \cup \Gamma_S^-\Big\} = 1.
  \end{aligned}
  \end{equation}
\end{enumerate}
\end{assumption}
\noindent
Note that species $S\in\mathcal S_\circ$ is consumed at rate $O(N^{\varphi_S})$ by \eqref{eq:varphi}, and is produced at rate $O(N)$ by \eqref{eq:prodcun}. Hence, the occupation measure of $S\in\mathcal S_\circ$ is $O(N^{1 - \varphi_S})$, which is the reason why we introduce $\hat\beta_R$ in \eqref{hatbeta}.
In other words, if $R\in\mathcal R_1$, it consumes species $S_R$, which is present only a fraction $O(N^{1-\varphi_{S_R}})$ of the time. This implies that the process counting how often reaction $R$ occured increases by one at rate $O(N^{\widetilde \beta_R + 1 - \varphi_{S_R}}) = O(N^{\widehat\beta_R})$.


\begin{example}[The simplified Hill dynamics~2]\label{ex:simHill2}
  Here, the fast sub-system, i.e.\ the sub-system only involving $E_1, E_3$, is first order, i.e.\ Assumption~\ref{ass:1}.2 applies. In \eqref{eq:CRNhill2}, we find $\widetilde \beta_1 = 1, \widetilde\beta_2 = \widetilde\beta_3 = \beta$. Here, we have that $\mathcal R = \mathcal R_1 = \mathcal R_1'$, since all reactions consume and produce some $S\in\mathcal S_\circ = \{E_1, E_3\}$. We find $\varphi_{E_1} = 1, \varphi_{E_3} = \beta$, and therefore $\widehat \beta_2 = \widehat\beta_3 = \beta + 1 - \beta = 1$, to the effect that $\psi_{E_1} = \psi_{E_3} = 1$, i.e.\ \eqref{eq:prodcun} holds. For $S=S_1, S_2$, we find that \eqref{eq:prodcun2} holds for the same reason. 
\end{example}

\begin{remark}[Some connections to previous work]
In the setting above, there are two aspects which lead to more general CRNs than in previously published papers: 
\begin{enumerate}
  \item In previous work of Kurtz and co-authors on multi-scale CRNs, fast species need to be produced and consumed at rates of the same order of magnitude. This can be seen from both, the balance condition (3.2) of \cite{KangKurtz2013} and the prerequisites of the averaging condition (Condition 2.4) of \cite{KangKurtzPopovic2014}. 
  \\
  In \cite{KangKurtz2013}, it is required that the production rate of every species is on the same order as its consumption rate (see their (3.2)), or its abundance is at least comparable to the fastest reaction it is involved in (see their (3.3)). For the fast species from above, if $\varphi_S>1$, neither is the case. Condition~2.4 of \cite{KangKurtzPopovic2014} requires that there is a limiting operator $L_2$ (not depending on $N$) which describes the evolution of fast species. If these are consumed faster than they are produced, $L_2$ would depend on $N$, so the only possibility is that $\varphi_S=1$ (i.e.\ $\widehat\beta_R = \widetilde\beta_R$). So, the analysis from above is more general since $\varphi_S>1$ is allowed for fast species $S$.
  \item Recently, \cite{HoesslyWiufXia2021} extended the approach of intermediate species from \cite{CappellettiWiuf2016} to non-interacting species and the existence of a fast chain of reactions. As above, a  set of species is non-interacting if its reaction network is at most first order if all other species are held constant. Assuming that the set of fast species and non-interacting species are equal, another restriction is imposed. Precisely, the existence of a fast chain of reaction requires that every fast species is created only via slow species, which does not allow for conserved quantities within the fast species. As we will discuss below, this is possible in our setting; e.g.\ as discussed in Example~\ref{ex:simHill1}, $Z_1 + Z_3$ is conserved.
\end{enumerate}

\end{remark}

\subsection{The Law of Large Numbers (LLN)}
\label{ss:33}
We wish to show that $V^N \xRightarrow{N\to\infty} V$ for some Markov process $V$. The goal of this section is to compute the possible generator $G$ of $V$. Let us consider some smooth $f$ only depending on $v$. Note that $Gf$ may only depend on $v$ (and not on $z$). With the above Assumption~\ref{ass:1}, the terms $G_0^Nf$ and $G_2^Nf$ in \eqref{eq:CRNgen} can be split in sums over $R\in \mathcal R_0$ and $R\in \mathcal R_1$. The generator terms corresponding to $R\in \mathcal R_0$ do not depend on $z$, but the terms corresponding to $R\in \mathcal R_1$ do. Therefore, in order to obtain $Gf$, we are looking for $g_N$ such that $G_0^Nf + G_2^N g_N$ only depends on $v$, and quantities which are constant for the evolution. We make the ansatz
\begin{align}
   \label{eq:g1}
   g_N(v,z) = 
   \nabla_\bullet f(v) \cdot a_N(v) \cdot z 
\end{align}
for some $a_N = a_{\mathcal S_\bullet\mathcal S_\circ} = (a_{SS'})_{S\in\mathcal S_\bullet S'\in\mathcal S_\circ}$ with $a_{SS'} \in \mathcal C^\infty(\mathbb R_+^{\mathcal S_\bullet})$, $S\in\mathcal S_\bullet, S'\in\mathcal S_\circ$. We obtain from \eqref{eq:CRNgen}
\begin{equation}
    \label{eq:feq}
\begin{aligned} 
    (G_0^Nf + G_2^Ng_N)(v,z) & = \sum_{R\in\mathcal R_0}  \kappa_R' v_{\mathcal S_\bullet}^{\nu_{R\mathcal S_\bullet}} \nabla_\bullet f(v) \cdot (\zeta_{R\mathcal S_\bullet} + a_{\mathcal S_\bullet \mathcal S_\circ}(v) \cdot \zeta_{R\mathcal S_\circ})
    \\ & \quad + \sum_{R\in\mathcal R_1}\kappa_R' z_{S_R} v_{\mathcal S_\bullet}^{\nu_{R\mathcal S_\bullet}} \nabla_\bullet f(v) \cdot (\zeta_{R\mathcal S_\bullet} + a_{\mathcal S_\bullet \mathcal S_\circ}(v) \cdot \zeta_{R\mathcal S_\circ}) + \varepsilon_f^N(v,z)
\end{aligned}
\end{equation}
and we note that it may be possible to choose $a_N$ such that all terms proportional to all $z_{S_R}$'s in the sum over $\mathcal R_1$ on the right hand side vanish. Another option is that there conserved linear combinations of fast species, i.e.\ some set $\mathcal C$ and $\xi_{\mathcal C} = (\xi_C)_{C\in\mathcal C}$ with $\xi_C \in \mathbb Z^{\mathcal S_\circ}$ all linearly independent, such that $\xi_C\cdot \zeta_{R\mathcal S_\circ} = 0$ for all $R\in\mathcal R$. In this case, $z\cdot \xi_C$ is a constant under the evolution, and there is a chance that we pick $a_{\mathcal S_\circ\mathcal S_\bullet}$ such that the right hand side only depends on $z$ via $z\cdot\xi_{\mathcal C} := (z\cdot\xi_C)_{C\in\mathcal C}$. 

In the sequel, we will assume that $g_N(v,z)$ is chosen such that
there is some smooth $\ell_{\bullet}$, taking values in $\mathbb R^{\mathbb S_\bullet}$ with
\begin{equation}\label{eq:g2}
  \begin{aligned}
    (G_0^Nf + G_2^Ng)(v,z) & = H^Nf(v) + \varepsilon_f^N(v,z), \\
    H^Nf(v) & = \sum_{R\in\mathcal R_0} \kappa_R' v_{\mathcal S_\bullet}^{\nu_{R\mathcal S_\bullet}} \nabla_\bullet f(v) \cdot (\zeta_{R\mathcal S_\bullet} + a_{\mathcal S_\bullet \mathcal S_\circ}(v) \cdot \zeta_{R\mathcal S_\circ})
    \\ & \qquad \qquad \qquad \qquad \qquad + \nabla_\bullet f(v) \cdot \ell_\bullet(\kappa'_{\mathcal R}, z\cdot \xi_{\mathcal C}, v).
  \end{aligned}
\end{equation}
Note that $z\cdot \xi_C$ is constant, so we can use $\ell_\bullet$ with $Z^N(0) \cdot \xi_C$ as a global constant. Moreover, since $\kappa'$ still depends on $N$, the right hand side depends on $N$, but may have a proper limit as $N\to\infty$, which we will denote $Gf(v)$. Then, if $(V^N)_{N=1,2,...}$ satisfies the compact containment condition and \eqref{eq:T1a}--\eqref{eq:T1d} hold (to be discussed in Section~\ref{S:ccc}), a weak limit of  $V^N$ solves the $(G, \mathcal D(G))$ martingale problem. 

\begin{example}[The simplified Hill dynamics~3]
We will now make the generator calculations leading to the LLN for \eqref{eq:CRNhill2}. Note that $M=Z_1^N + Z_2^N$ does not change over time, i.e.\ the limiting generator may still depend on $M$. Provided the required assumptions (compact containment of $(V^N)_{N=1,2,...}$ and \eqref{eq:T1a}--\eqref{eq:T1d} hold), we will show the following:\\
Let $v\in\mathbb R_+$, 
\begin{align}\label{eq:pNvhill2}
p^N(v):= -\frac{M\kappa_1\kappa_3' v^2}{\kappa_2' + \kappa_3' + \kappa_1 v^2} \xrightarrow{N\to\infty} p(v) := \begin{cases} \displaystyle -\frac{M\kappa_1\kappa_3 v^2}{\kappa_2 + \kappa_3 + \kappa_1 v^2}, & \beta = 1, \\ \displaystyle -\frac{M\kappa_1\kappa_3 v^2}{\kappa_2 + \kappa_3}, & \beta > 1\end{cases}
\end{align}
and $V$ the solution of the ODE
\begin{align}
  \label{eq:simHillV}
  dV = p(V) dt, \qquad V_0=v.
\end{align} 
Then, letting $V^N$ as above and if $V^N_0 \xRightarrow{N\to\infty} v$, then  $V^N \xRightarrow{N\to\infty} V.$ 

~

\noindent In the case $\beta=1$, this result can be proved using the technique of stochastic averaging: Fixing the amount of slow species, we see that the fast network reduces to $\operatorname{\text{$E_1$}} \autorightleftharpoons{\text{$N\kappa_{1}v^2$}}{\text{$N(\kappa_{2} + \kappa_3)$}}   \operatorname{\text{$E_3$}}$. In its equilibrium, $E_3$ is binomially distributed with $M$ and $\frac{\kappa_1 v^2}{\kappa_1 v^2 + \kappa_2 + \kappa_3}$. From this, the LLN can be derived, e.g.\ by applying the main result from \cite{Kurtz1992}. However, it is unclear if the same technique gives a result for $\beta > 1$. Here, we follow the route from above, which will in addition pay dividends when deriving the CLT in the next subsection.

~

\noindent
For the proof, take some smooth $f$ depending only on $v$ and find $g_N$ such that  $G_0^Nf+G_2^Ng - \varepsilon_N^f$ only depends on $v$. Following \eqref{eq:g1}, we take $g_N(v,z)=za_N(v) f'(v)$. We find (abbreviating $a_N := a_N(v)$)
\begin{align*}
  (G_0^N f & + G_2^N g_N)(v,z) - \varepsilon_f^N(v,z) \\ & = M(\kappa_2' (2 + a_N) + \kappa_3' (1+a_N)) f'(z) \\ & \qquad \qquad \qquad \qquad \qquad + z\big(-\kappa_1 v^2(2 + a_N) - \kappa_2'(2 + a_N) - \kappa_3' (1+a_N)\big) f'(v).
\end{align*}
Setting
\begin{align}
    \label{eq:aNMhill}
    a_N = - \frac{2\kappa_1 v^2 + 2\kappa_2' + \kappa_3'}{\kappa_2' + \kappa_3' + \kappa_1 v^2} = -1 - \frac{\kappa_1 v^2 + \kappa_2'}{\kappa_2' + \kappa_3' + \kappa_1 v^2} = -2 + \frac{\kappa_3'}{\kappa_2' + \kappa_3' + \kappa_1 v^2},
\end{align}
this leads to
\begin{align}
    \notag 
    H^Nf(v) := (G_0^N f + G_2^N g_N)(v,z) - \varepsilon_f^N(v,z) = -M\frac{\kappa_1\kappa_3' v^2}{\kappa_2' + \kappa_3' + \kappa_1 v^2}f'(v) = p^N(v) f'(v).
\end{align}
Taking $N\to\infty$ then shows the result.
\end{example}
~
\subsection{The functional Central Limit Theorem (CLT)}\label{ss:34}
Next, let us study fluctuations, i.e.\ we are going to study the limit of $U^N := N^{1/2}(V^N - W^N)$, where $W^N$ has generator $H^N$ (recall from \eqref{eq:g2}). For this, we study the generator for the Markov process $(U^N, Z^N, W^N)$, where $(U^N, W^N)$ are slow and $Z^N$ is fast. When doing so, we have to exchange $V^N = W^N + N^{-1/2} U^N$ in \eqref{eq:CRNgen}. In particular, writing $v = V^N, u = U^N, w = W^N$,
\begin{align*}
  v_{[k],N^{-1}} & = 
  (w + N^{-1/2}u)(w + N^{-1/2}u - N^{-1}) \cdots (w + N^{-1/2}u - (k-1)N^{-1})
  \\ & = w^k + N^{-1/2}kw^{k-1}u + N^{-1} \binom k2 w^{k-2}(u^2-w) + N^{-3/2} \delta^N(u,w,k)
\end{align*}
where $\delta^N(u,w,k)$ is defined appropriately. Now, assume that $f$ only depends on $(u,z)$. Looking at \eqref{eq:CRNgen}, we can write the generator (denoting by $\nabla_\bullet$ the derivative with respect to $u$)
\begin{equation}\label{eq:genL}
  \begin{aligned}
    L^Nf(u,z,w) & = \sum_{R\in\mathcal R} N\kappa_R' z_{[\nu_{R\mathcal S_\circ}]}(w + N^{-1/2}u)_{[\nu_{R\mathcal S_\bullet}], N^{-1}} (f(u + N^{-1/2}\zeta_{R\mathcal S_\bullet}, z + \zeta_{R\mathcal S_\circ}) - f(u,z)) \\ & \qquad \qquad \qquad \qquad \qquad \qquad\qquad \qquad \qquad \qquad \qquad \qquad - N^{1/2} H^Nf(u,z,w)
    \\ & = (L_{0}^N + N^{1/2}L_{1}^N + NL_{2}^N)f(u,z,w),
    \\ L_{0}^Nf(u,z,w) & = \sum_{R\in\mathcal R} \kappa_R'z_{[\nu_{R\mathcal S_\circ}]} \big((\nu_{R\mathcal S_\bullet}w^{\nu_{R\mathcal S_\bullet}-1} u) \nabla_\bullet f(u,z + \zeta_{R\mathcal S_\circ}) \cdot \zeta_{R\mathcal S_\bullet} \\ & \qquad + \tfrac 12 w^{\nu_{R\mathcal S_\bullet}} \zeta_{R\mathcal S_\bullet}^\top \cdot \nabla_\bullet^2 f(u,z + \zeta_{R\mathcal S_\circ}) \cdot \zeta_{R\mathcal S_\bullet}\big) \\ & \qquad + \binom{\nu_{R\mathcal S_\bullet}}2 w^{\nu_{R\mathcal S_\bullet-2}}(u^2 - w) (f(u, z + \zeta_{R\mathcal S_\circ}) - f(u,z))+ \varepsilon_{f}^N(u,z,w)\\ L_{1}^Nf(u,z,w) & = \sum_{R\in\mathcal R} \kappa_R'z_{[\nu_{R\mathcal S_\circ}]} \Big( w^{\nu_{R\mathcal S_\bullet}} \nabla_\bullet f(u, z + \zeta_{R\mathcal S_\circ}) \cdot \zeta_{R\mathcal S_\bullet} \\ & \qquad \qquad \qquad \quad + (\nu_{R\mathcal S_\bullet}w^{\nu_{R\mathcal S_\bullet}-1} u)(f(u, z+\zeta_{R\mathcal S_\circ}) - f(u,z))\Big) - H^Nf(u,z,w)
    \\ L_{2}^Nf(u,z,w) & = \sum_{R\in\mathcal R} \kappa_R'z_{[\nu_{R\mathcal S_\circ}]} w^{\nu_{R\mathcal S_\bullet}}(f(u, z+\zeta_{R\mathcal S_\circ}) - f(u,z)), \\ H^Nf(u,z,w) & = \sum_{R\in\mathcal R_0} \kappa_R' w_{\mathcal S_\bullet}^{\nu_{R\mathcal S_\bullet}} \nabla_\bullet f(u,z) \cdot( \zeta_{R\mathcal S_\bullet} + a_{\mathcal S_\bullet \mathcal S_\circ}(w)\cdot \zeta_{R\mathcal S_\circ}) + \nabla_\bullet f(u,z) \cdot \ell_\bullet(\kappa_{\mathcal R}', z\cdot\xi_{\mathcal C}, w),
    \\ |\varepsilon_f^N(u,z,w)| & \leq \sum_{R\in\mathcal R} N^{-1/2} \kappa_R' z_{[\nu_{R\mathcal S_\circ}]} \big(\tfrac 12 \nu_{R\mathcal S_\bullet} w^{\nu_{R\mathcal S}-1} u || \nabla^2 f||_\infty + \binom{\nu_{R\mathcal S_\bullet}}{2} w^{\nu_{R\mathcal S_\bullet}-2}(u^2 + w) || \nabla_\bullet f||_\infty \\ & \qquad \qquad \qquad \qquad \qquad + \tfrac 16 w^{\nu_{R\mathcal S_\bullet}} \zeta_{R\mathcal S_\bullet}^\top \cdot \zeta_{R\mathcal S_\bullet}|| \nabla^3_\bullet f||_\infty + \delta^N(u,w,\nu_{R\mathcal S_\bullet}) ||f||_\infty\big).
  \end{aligned}
\end{equation}
Now, let $f$ depend only on $u$ (hence $L_2^Nf = 0$) and 
\begin{align}
    \notag 
    g_N(u,z) = \nabla_\bullet f(u) \cdot a_N(w) \cdot z
\end{align}
with $a_N$ as in \eqref{eq:g1} such that \eqref{eq:g2} holds. Then, looking at \eqref{eq:CRNgen} and \eqref{eq:genL}, we see by the similarity between $G_0^N$ and $L_1^N$ (noting the additional term $-H^N$ in $L_1^N$), and the similarity between $G_2^N$ and $L_2^N$ that 
$$L_1^Nf + L_2^N g_N = 0.$$ 
Then, we are looking for $h_N$ such that $L_0^Nf + L_1^Ng_N + L_2^Nh_N$ only depends on $(u,w)$, and the linear combinations $z\cdot\xi_{\mathcal C}$ which are constant in time. We make the ansatz
\begin{align}\label{eq:hNgen}
    h_N(u,z,w) & = \nabla_\bullet f(u) \cdot b_N(u,w) \cdot z + \nabla^2_\bullet f(u) \cdot (c_N(u,w) \cdot z + z^\top \cdot d_N(u,w) \cdot z) 
\end{align}
for some functions $b_N = (b_{SS'})_{S\in \mathcal S_\bullet, S'\in\mathcal S_\circ}$, $c_N = (c_{S, S', S''})_{S,S'\in\mathcal S_\bullet, S'' \in \mathcal S_\circ}, d_N = (d_{SS', S''S'''})_{S, S' \in \mathcal S_\bullet, S'', S''' \in \mathcal S_\circ}$, which depend on $N,w,u$. Then, since the fast subsystem is first order, we can write
\begin{equation}
    \label{eq:GCRN2}
    \begin{aligned}
        (& L_0^Nf + L_1^Ng_N + L_2^Nh_N)(u,z,w) \\ & = \sum_{R\in \mathcal R} \kappa_R'z_{[\nu_{R\mathcal S_\circ}]} \Big((\nu_{R\mathcal S_\bullet} w^{\nu_{R\mathcal S_\bullet}-1} u)  \nabla_\bullet f(u) \cdot (\zeta_{R\mathcal S_\bullet} + a_N(w) \cdot \zeta_{R\mathcal S_\circ} + w^{\nu_{R\mathcal S_\bullet}}\nabla_\bullet f(u) \cdot b_N(u,w) \cdot \zeta_{R\mathcal S_\circ} ) \\ & \qquad +  w^{\nu_{R\mathcal S_\bullet}} \zeta_{R\mathcal S_\bullet}^\top \cdot \nabla_\bullet^2 f(u) \cdot (\tfrac 12 \zeta_{R\mathcal S_\bullet} + a_N(w) \cdot (z + \zeta_{R\mathcal S_\circ})) \\ & \qquad + w^{\nu_{R\mathcal S_\bullet}}\Big(\nabla^2_\bullet f(u) \cdot(c_N(u,w) \cdot \zeta_{R\mathcal S_\circ} \\ & \qquad \qquad + (\zeta_{R\mathcal S}^\top \cdot d_N(u,w) \cdot (z + \zeta_{R\mathcal S_\circ}) + z^\top \cdot d_N(u,w) \cdot \zeta_{R\mathcal S}^\top )  \Big)\Big) - H^Ng_N(u,z,w) + \varepsilon_f^N(u,z,w)
    \end{aligned}
\end{equation}
Then, splitting the sum in \eqref{eq:GCRN2} into $R\in\mathcal R_0$ and $R\in\mathcal R_1$, we see that for $R\in\mathcal R_1$ there are terms proportional to $z_{S_R} \nabla_\bullet f(u)$, and there is a chance to choose $b_N$ such that these terms vanish or depend only on $(z\cdot \xi_C)_{C\in\mathcal C}$. In a second step, we choose $c_N$ and $d_N$ such that the remaining terms, all of which are proportional to $\nabla_\bullet^2 f$, vanish or depend only on $(z\cdot \xi_C)_{C\in\mathcal C}$. Assuming that $g_N,h_N$ have proper limits as $N\to\infty$, this leaves us with 
\begin{align*}
    Lf(u,w) := \lim_{N\to\infty} (L_0^N f + L_1^N g_N + L_2^Nh_N)(u,z,w) - \varepsilon_f^N(u,z,w),
\end{align*}
where $L$ is the generator of the limiting process. Provided the compact containment condition of $(U^N)_{N=1,2,...}$ and \eqref{eq:T1a}--\eqref{eq:T1d} hold, this shows convergence with Theorem~\ref{T1}.

\begin{example}[The simplified Hill dynamics 4]
Let $u\in\mathbb R$, $p^N$ and $p$ as in \eqref{eq:pNvhill2}, $V$ as in \eqref{eq:simHillV}, and $W^N$ the solution of the ODE
$$ dW^N = p^N(W^N) dt, \qquad W_0=v.$$
(So, $V=W^N$ for $\beta=1$.) Moreover, let $U$ be the solution of the SDE
\begin{equation}\label{eq:simHillU}
\begin{aligned}
  dU & = \mu(U,V) dt + \sigma(V) dB, \\ \mu(U,V) & = \begin{cases}\displaystyle -2M \frac{\kappa_1(\kappa_2+\kappa_3)\kappa_3 UV}{(\kappa_2 + \kappa_3 + \kappa_1 V^2)^2}, \qquad \qquad \qquad \qquad  & \beta = 1, \\ \displaystyle p'(V)U, & \beta > 1, \end{cases} \\ \sigma^2(V) & = \begin{cases}\displaystyle \frac{(\kappa_1^2 V^4+2\kappa_1\kappa_2 V^2+(\kappa_2+\kappa_3)^2)M\kappa_1\kappa_3 V^2}{(\kappa_2 + \kappa_3 + \kappa_1 V^2)^3}, & \beta = 1, \\ \displaystyle |p(V)|, & \beta > 1 \end{cases} \\ 
\end{aligned}
\end{equation}
for some Brownian motion $B$. Then, letting $V^N$ as in Example~\ref{ex:simHill2}, and $U^N = N^{1/2}(V^N - W^N)$. If $U^N_0 \xRightarrow{N\to\infty} u$, then
$U^N \xRightarrow{N\to\infty} U.$

~

\noindent Consider the generator of the Markov process $(U^N,Z^N,W^N)$, which is for $f\in\mathcal C_c^\infty(\mathbb R^3)$, only depending on $(u,z)$, 
\begin{align*}
    L^Nf(u,z,w)&=\kappa_1N^{-1}(Nw+N^{1/2}u)(Nw+N^{1/2}u-1)z(f(u-2N^{-1/2},z-1)-f(u,z))\\
    &\qquad +N\kappa_2'(M-z)(f(u+2N^{-1/2},z+1)-f(u,z))\\
    &\qquad+N\kappa_3'(M-z)(f(u+N^{-1/2},z+1)-f(u,z))\\
    &\qquad -N^{1/2}p^N(w)f'(u,z) \\
    &=L_0^Nf(u,z)+N^{1/2}L_1^Nf(u,z)+NL_2^Nf(u,z),\\
    L_0^Nf(u,z)&=2\kappa_1w^2zf''(u,z-1)-4 \kappa_1wu zf'(u,z-1)+\kappa_1 (u^2-w) z(f(u,z-1)-f(u,z))\\
    &\qquad +2\kappa_2'(M-z)f''(u,z+1) +\tfrac 12 \kappa_3'(M-z)f''(u,z+1) + \varepsilon_f^N(u,z,w),\\
    L_1^Nf(u,z)&=-2\kappa_1w^2zf'(u,z-1)+2\kappa_1wuz(f(u,z-1)-f(u,z))\\
    &\qquad +2\kappa_2'(M-z)f'(u,z+1)+\kappa_3'(M-z)f'(u,z+1)-p^N(w)f'(u,z),\\
    L_2^Nf(u,z)&=\kappa_1w^2z(f(u,z-1)-f(u,z))+(\kappa_2'+\kappa_3')(M-z)(f(u,z+1)-f(u,z)), \\ |\varepsilon_f^N(u,z,w)| & \leq N^{-1/2} \big( (\tfrac 23 \kappa_1 w^2 z + \tfrac 23 \kappa_2'(M-z) + \tfrac 16 \kappa_3'(M-z)) ||f'''||_\infty + 2\kappa_1 uwz ||f''||_\infty \\ & \qquad \qquad \qquad \qquad \qquad \qquad \qquad \qquad \qquad \qquad \qquad +  \kappa_1(u^2 + w)||f||_\infty\big).
\end{align*}
With $f$ only depending on $u$ and $g_N(u,z,w)=za_N(w)f'(u)$ as above, we have that $L_1^Nf+L_2^Ng_N=0$. Then we use the ansatz
$$h_N(u,z,w)=z b_N  f'(u)+\Big( z c_N +  \binom{z}{2} d_N \Big)f''(u),$$
for some functions $b_N, c_N, d_N$ (which might depend on $N, u, w$). We receive at (writing $a:=a_N, b:=b_N, c:=c_N, d:=d_N$)
\begin{align*}
    (L_0^N f & + L_1^N g_N + L_2^N h_N)(u,z,w) - \varepsilon_f^N(u,z,w)
    \\ & = \Big( M(\kappa_2' + \kappa_3')b - z(4\kappa_1 wu + 2\kappa_1 wu a + ( \kappa_1 w^2 + \kappa_2' + \kappa_3') b) \Big) f'(u) 
    \\ & \qquad + \Big(M(2\kappa_2'(1+a) + \tfrac 12 \kappa_3'(1 + 2a) + M(\kappa_2' + \kappa_3') c ) 
    \\ &  \qquad +   z(2\kappa_1 w^2 - 2\kappa_2' - \tfrac 12 \kappa_3' - 2\kappa_1 wu a - (2\kappa_2' + \kappa_3')(M-2) a - p^N(w) a \\ & \qquad \qquad \qquad \qquad \qquad \qquad \qquad \qquad - ( \kappa_1 w^2 + \kappa_2' + \kappa_3') c+  (M-1)(\kappa_2' + \kappa_3')d ) 
    \\ & \qquad + z(z-1)(-2\kappa_1 w^2 a - (2\kappa_2' + \kappa_3') a) - (\kappa_1 w^2 + \kappa_2' + \kappa_3') d \Big)f''(u). 
\end{align*}
With $a_N$ from \eqref{eq:aNMhill}, we immediately see that only through
\begin{align*}
    b & = - \frac{2\kappa_1 wu(2+a)}{\kappa_1 w^2 + \kappa_2' + \kappa_3'} - = -\frac{2\kappa_1 \kappa_3' wu}{(\kappa_1 w^2 + \kappa_2' + \kappa_3')^2}
\end{align*}
the term proportional to $zf'(u)$ on the right hand side vanishes. Then, using that $p^N(w) = M\kappa_2'(2 + a) + M\kappa_3'(1+a)$, finding the roots of the terms proportional to $zf''(u)$ and $z(z-1)f''(u)$ on the right hand side for $c,d$, we find 
\begin{align*}
    (L_0^Nf + & L_1^Ng_N+  L_2^Nh_N)(u,z,w) - \varepsilon_f^N(u,z,w)\\
    &=-2M \frac{\kappa_1(\kappa_2'+\kappa_3')\kappa_3' uw}{(\kappa_2' + \kappa_3' + \kappa_1 w^2)^2}f'(u) + \frac 12 \frac{(\kappa_1^2w^4+2\kappa_1\kappa_2'w^2+(\kappa_2'+\kappa_3')^2)M\kappa_1\kappa_3'w^2)}{(\kappa_2' + \kappa_3' + \kappa_1 w^2)^3}f''(u).
\end{align*}
Since $W^N\xRightarrow{N\to\infty} V$, the limit of the right hand side for $N\to\infty$ gives the generator of $U$ and we are done.
\end{example}

\subsection{Checking the conditions of Theorem~\ref{T1}}\label{S:ccc}
If the calculations from the last section go through, we still have to verify the conditions of Theorem~\ref{T1}, i.e.\ the compact containment condition of $(V^N)_{N=1,2,...}$ for the LLN and $(U^N)_{N=1,2,...}$ for the CLT, as well as the corresponding conditions \eqref{eq:T1a}--\eqref{eq:T1d}. Checking these might depend on the concrete example, but here, we give a general strategy. In all steps, note that we require the fast network to be first order. This means that either~(i) fast species only appear on both sides of each $R\in\mathcal R$ of the fast network or~(ii) one fast species $S \in \mathcal S_\circ$ is produced only from slow species. (The case where one fast species produces some slow species only appears in (ii), since otherwise fast species will quickly be consumed.) In (i), there are conserved quantities in the fast network, i.e.\ (using the notation from above) there is $\xi_C$ such that $z\cdot \xi_C$ is constant. In (ii), the discussion after \eqref{eq:prodcun2} applies: the fast species $S$ is created at some rate $C_0^N N$, and is consumed at rate $C_1^N N^\beta Z^N_S$ for $\beta \geq 1$. For such processes, Lemma~\ref{l:ccc} is useful for bounding moments. The number of other fast species will be bounded by $Z_S^N$.

As for \eqref{eq:T1a}, recall from \eqref{eq:g1} and \eqref{eq:hNgen}, that both, $g_N$ and $h_N$ are polynomials in the fast species. So,  from \eqref{eq:ccc2},
$$N^{-1/2} \mathbb E[\sup_{0\leq t\leq T} (Z_S^N(t))^k] \xrightarrow{N\to\infty} 0, \qquad T>0, k=1,2,...$$ provides the required argument for \eqref{eq:T1a} to hold if $f$ has compact support and $a_N, b_N, c_N, d_N$ are continuous. For the integrals in \eqref{eq:T1b} and \eqref{eq:T1d}, we use \eqref{eq:ccc1} for 
$$\sup_{N\in\mathbb N} \mathbb E\Big[\int_0^T N^{\beta-1}(Z_S^N(t))^k dt\Big] < \infty, \qquad T>0, k=1,2,...$$ if the corresponding terms again come as a polynomial of fast species. So, we are left with showing the compact containment condition for  $(V^N)_{N=1,2,...}$ and $(U^N)_{N=1,2,...}$. Focusing on $(U^N)_{N=1,2,...}$, we start by the same generator calculations as in the last section, using $f(u) = u_S$ for some $S\in\mathcal S_\bullet$. This gives some $g_N$ and $h_N$ such that
\begin{align*}
L^N & (f + N^{-1/2} g_N + N^{-1} h_N) \\ & =  N^{1/2} (L_1^N f + L_2^N g_N) + (L_0^N f + L_1^N g_N + L_2^N h_N) + N^{-1/2} (L_0^N g_N + L_1^N h_N) + N^{-1} L_0^N h_N \\ & = 
Lf + \varepsilon^N_f + N^{-1/2} (L_0^N g_N + L_1^N h_N) + N^{-1} L_0^N h_N.  
\end{align*}
Then, provided that \eqref{eq:T1a}--\eqref{eq:T1d} also hold for $f(u) = u_S$, we have that
\begin{align*}
  U_S^N(t) & = M^N_S(t) - N^{-1/2} g_N(U_t^N, Z_t^N, W_t^N) - N^{-1}h_N(U_t^N, Z_t^N, W_t^N)  \\ & \qquad \qquad +  \int_0^t (Lf + \varepsilon^N_f + N^{-1/2}(L_0^N g_N + L_1^N h_N) + N^{-1} L_0^N h_N)(U_s^N, Z_s^N, W_s^N)ds
\end{align*}
for the martingale $M^N_S := (M^N_S(t))_{t\geq 0}$ with 
\begin{align*}
    M^N_S(t) = U_S^N(t) + (N^{-1/2} g_N + & N^{-1} h_N)(U_t^N, Z_t^N, W_t^N) \\ & - \int_0^t L^N(f + N^{-1/2} g_N + N^{-1} h_N)(U^N_s, Z_s^N, W_s^N)ds.
\end{align*}
Then, if we can bound $\sup_{0\leq t\leq T} |M_S^N(t)|$ (e.g.\ by bounding the quadratic variation of this martingale) uniformly in $N$, and if $Lf$ grows at most linearly in $U^N$, we see using \eqref{eq:T1a}--\eqref{eq:T1d} that there are bounded functions $a_S, b_S \geq 0$ such that 

\begin{align}
    \label{eq:gron}
     \mathbb{E}[\sup_{s\leq t}|U_S^N(s)|]\leq b_S(t)+\int_0^t a_S(s)\cdot \mathbb{E}[\sup_{r\leq s}|U^N(r)|]ds.    
\end{align}
With the Gronwall lemma, we obtain $\mathbb{E}[\sup_{s\leq t}|U^N(t)|]\leq b_S(t)+\int_0^t b_S(s)a_S(s)e^{\int_s^t a_S(r)dr}ds$ and Markov's inequality shows the compact containment condition for $(U_S^N)_{N=1,2,...}$.

\begin{example}[The simplified Hill dynamics~5]
  \label{ex:simHill5} While we provided the necessary generator calculations needed for convergence $V^N\xRightarrow{N\to\infty} V$ with $V$ as given in \eqref{eq:simHillV} and $U^N\xRightarrow{N\to\infty} U$ with $U$ as in \eqref{eq:simHillU}, we still need to check the corresponding conditions \eqref{eq:T1a}--\eqref{eq:T1d} and the comopact containment conditions. We note that in this example, for $T,k>0$, using Lemma~\ref{l:ccc}, 
  \begin{align}\label{eq:simHillfci}
  \mathbb E\Big[\int_0^T (M - Z^N_t)^k dt\Big] & = O(N^{1-\beta}).
  \end{align}
  In addition $0\leq Z^N\leq M$, so $\limsup_{N\to\infty} N^{-1/2} Z^N = 0$. These properties already imply \eqref{eq:T1a}--\eqref{eq:T1d} in both, the LLN and the CLT. For the compact containment condition in the LLN, note that $V^N + 2Z^N/N$ is a decreasing process, which already implies compact containment (provided $V_0^N$ converges). The hardest work is required by the compact containment condition in the CLT. For this, let $g_N(t) := g_N(Z^N_t, W^N_t) = Z^N_t a_N(W^N_t)$ and $h(t) := h_N(U^N_t, Z^N_t, W^N_t) = Z^N_t b_N(U^N_t, W^N_t)$ (i.e.\ as above with $f(u) = u$). Then, \eqref{eq:T1a}--\eqref{eq:T1d} also hold in this case (since the $(u^2+w)$-term in $\varepsilon_f^N$ vanishes in this case, leading to $\varepsilon_f^N = 0$). Then,
  \begin{align*}
    U^N_t & = M^N_t - N^{-1/2}g_N(t) - N^{-1}h_N(t) + \int_0^t 2M \frac{\kappa_1(\kappa_2'+\kappa_3')\kappa_3' U^N_s W_s}{(\kappa_2' + \kappa_3' + \kappa_1 W_s^2)^2} ds,
    \\ M^N_t&:=U^N_t+N^{-1/2}g_N(t) + N^{-1} h_N(t) - \int_0^t 2M \frac{\kappa_1(\kappa_2'+\kappa_3')\kappa_3' U^N_sW_s}{(\kappa_2 + \kappa_3' + \kappa_1 W_s^2)^2}ds,
  \end{align*}
  where, $(M^N_t)_{t\geq 0}$ is a martingale with quadratic variation
  \begin{align*}
    [M^N]_t&=N[V^N+N^{-1}g_N]_t\\
    & \approx \int_0^t (2 + a_N(W_s))^2(\kappa_1 (V^N_s)^2Z^N_s + \kappa_2'  (M-Z^N_s))  +  (1+a_N(W_s))^2 \kappa'_3 (M-Z^N_s)ds.
  \end{align*}
  Since $V^N,Z^N,W$ are bounded, $g_N$ and $h_N$ are bounded as well. Moreover, from \eqref{eq:simHillfci} we see that the quadratic variation of $M^N$ is locally bounded, so we can conclude that there are $a_S$ and $b_S$ such that \eqref{eq:gron} holds, and the compact containment condition for $(U^N)_{N=1,2,...}$ follows.    
\end{example}

\section{Examples}\label{S:examples}
In this section we give some more concrete applications of our theory. We start with Michaelis-Menten kinetics in Section~\ref{ss:MM}, i.e.\ we rederive the results from  \cite{KangKurtzPopovic2014} using the approach from Section~\ref{S:CRNs}. These kinetics is similar to the simplified Hill dynamics from Examples~\ref{ex:simHill1}--\ref{ex:simHill5}. Still, we extend results from \cite{KangKurtzPopovic2014} by allowing that the balance condition for the fast species to  fail, and from \cite{CappellettiWiuf2016} by proving the CLT. The simplified Hill example is extended in Section~\ref{ss:Hill} to account for separate effects of binding of the two ligands to the macromolecule. In Section~\ref{ss:ABCD}, we are dealing with an example using two fast species where the number of molecules of the fast species is unbounded. For some parameter combinations ($\beta, \gamma>1$), the fast species are called intermediate by \cite{CappellettiWiuf2016}, so we extend their results by the CLT. Then, in Section~\ref{ss:CWmain}, we are dealing with the main example in \cite{CappellettiWiuf2016}, but again add to their results the CLT.


We note that not all calculations were carried out by hand. The approach that the right hand sides of \eqref{eq:feq} and  \eqref{eq:GCRN2} do not depend on the fast variables leads to linear systems of equations with up to 15 equations in Section~\ref{ss:CWmain}. Therefore, we produced an ancillary file with a {\tt sagemath}-commands \cite{sagemath}, which can be found on {\tt arxiv}. This file contains the computations not printed, but necessary for Examples~\ref{ex:simHill1}--\ref{ex:simHill5} and for Sections~\ref{ss:MM}, \ref{ss:Hill}, \ref{ss:ABCD} and \ref{ss:CWmain}.

\subsection{Michaelis-Menten kinetics}
\label{ss:MM}
Our first example is the famous Michaelis-Menten kinetics. It arises for the Chemical Reaction Network \begin{align}\label{M:MM}
  \operatorname{\text{$E_1$}} + \operatorname{\text{$S_1$}}
  \autorightleftharpoons{\text{$\kappa_1$}}{\text{$N^\gamma\kappa_2$}} \operatorname{\text{$E_2$}} \autorightarrow{\text{$N^\gamma \kappa_{3}$}}{} \operatorname{\text{$E_1$}} + \operatorname{\text{$S_2$}}
\end{align}
for $\gamma\geq 1$, where $S_1,S_2$ are in high abundance (order $N$) and $E_1, E_2$ are in low
abundance (order 1). The balance condition for $E_2$ only holds if $\gamma=1$, since $\gamma>1$ implies that $E_2$ is produced slower than it is consumed. Chemically, $E_2$ is a complex formed out of $E_1$ (enzyme) and $S_1$ (substrate), and $S_2$ is frequently described as product of the reaction. Denoting by $Z_1^N, Z^N_{2}$ the number of molecules $E_1$ and $E_2$, we note that $Z^N_1 + Z^N_{2} =:M$ is a constant. Further, we define $NV^N_1, NV^N_2$ to be the number of molecules of $S_1$ and $S_2$, and we find that $V^N_1 + V^N_2 + Z^N_{2}/N$ is a constant. So, since we will have that $Z_2^N = O(1)$, the dynamics of $V_2^N$ can be read off from the dynamics of $V_1^N$ (up to an $O(1/N)$-error). So, the system can be described through $V^N := V_1^N$ and $Z^N := Z_1^N$. We obtain the following.

\begin{proposition}[LLN and CLT for \eqref{M:MM}]
  Let $u\in\mathbb R$ and $v\in\mathbb R_+$, 
  \begin{align}\notag 
    p^N(v) :=-M \frac{N^{\gamma-1}\kappa_{3}\kappa_{1}v}{N^{\gamma-1}(\kappa_{2} + \kappa_{3}) + v\kappa_{1}} \xrightarrow{N\to\infty} 
    p(v) := \begin{cases} \displaystyle -\frac{M\kappa_{3}\kappa_{1}v}{\kappa_{2} + \kappa_{3} + v\kappa_{1}}, & \text{ for }\gamma = 1, \\[3ex] \displaystyle  -\frac{M\kappa_{3}\kappa_{1}v}{\kappa_{2} + \kappa_{3}}, & \text{ for }\gamma>1, \end{cases}    
  \end{align}
  $V$ and $W^N$ the solutions of the ODEs
  $$ dV = p(V) dt, \qquad dW^N = p^N(W^N) dt, \qquad W_0^N = V_0 =v.$$
  Moreover, let $U$ be the solution of the SDE
  \begin{align*}
     &dU = - \frac{M\kappa_1\kappa_3(\kappa_2 + \kappa_3) }{(\kappa_1 V + \kappa_2 + \kappa_3)^2} U dt + \sqrt{\frac{M\kappa_1 V \kappa_3 (\kappa_1 V + \kappa_2)^2}{2(\kappa_1 V + \kappa_2 + \kappa_3)^3} + \frac{M\kappa_3^2 \kappa_1 V(2\kappa_2 + \kappa_3)}{2(\kappa_1 V + \kappa_2+ \kappa_3)^3}} dB, &&& \text{ if }\gamma=1,
     \\&dU = p(U)dt + \sqrt{|p(V)|} dB, &&& \text{ if }\gamma >1,
  \end{align*}
  \sloppy for some Brownian motion $B$. Then, letting $V^N$ be as above, and $U^N = N^{1/2}(V^N - W^N)$. If $V^N_0 \xRightarrow{N\to\infty} v$ and $U^N_0 \xRightarrow{N\to\infty} u$, then
  $$ V^N_0 \xRightarrow{N\to\infty} V, \qquad U^N_0 \xRightarrow{N\to\infty} U. $$
\end{proposition}

\begin{remark}[Previous results on Michaelis-Menten kinetics]
  The above LLN has been obtained various times; see e.g.\ Section~6.4 of \cite{KangKurtz2013} for $\gamma=1$, or Example~4.5 of \cite{CappellettiWiuf2016} for $\gamma>1$. (They in fact consider two differnt scalings for $\kappa_2$ and $\kappa_3$.) In the case $\gamma=1$, the CLT has first been obtained in Section~5.2 of \cite{KangKurtzPopovic2014}. 
\end{remark}


\begin{proof}
We set $\kappa_2' = N^{\gamma-1} \kappa_2$ and $\kappa_3' = N^{\gamma-1} \kappa_3$. For the Markov process $X^N = (V^N, Z^N)$, the generator is given for $f \in \mathcal C_c^\infty(\mathbb R\times\mathbb R$, writing $f'$ for the derivative according to the first variable,
\begin{equation}
  \notag 
  \begin{aligned}
    G^Nf(v,z) & = \kappa_{1} Nvz(f(v-N^{-1},z-1) - f(v,z)) + \kappa'_{2} N(M-z) (f(v+N^{-1},z+1) - f(v,z)) \\ &
    \qquad \qquad \qquad \qquad \qquad \qquad \qquad + \kappa'_{3} N(M-z) (f(v, z+1) - f(v,z)) \\ & =
    G_0^Nf(v,z) + N G_2^Nf(v,z),
    \\
    G_0^Nf(v,z) & = -\kappa_{1} vz f'(v, z-1) +
    \kappa_{2}' (M-z)  f'(v, z+1) + \varepsilon_f^N(v,z),
    \\
    G_2^Nf(v,z) & = \kappa_{1} vz (f(v, z-1) - f(v,z)) + (\kappa_{2}' + \kappa_{3}') (M-z)(f(v,z+1) - f(v,z)), \\ |\varepsilon_f^N(v,z) | & \leq \tfrac 12 N^{-1} \big(\kappa_1 vz + \kappa_2' (M-z)) ||f''||_\infty.
  \end{aligned}
\end{equation}
We mention here that $0\leq Z^N(t) \leq M$ by construction and that for all $T,a,k$ (using similar arguments as in Lemma~\ref{l:ccc})
\begin{align}\label{eq:fci}
  \mathbb E\Big[\int_0^T (M - Z^N_t)^k dt\Big] & = O(N^{1-\gamma}).
\end{align}
1. For the LLN of $V^N$ as $N\to\infty$, take $f\in\mathcal C_c^\infty(\mathbb R)$. Our task is to look for $g_N$ (depending on $v,z$) such that $G_0^N f + G_2^N g_N$ only depends on $v$. Choosing 
\begin{align} \notag
    g_N(v,z) & = za_N(v) f'(v),
\end{align}
for some (function) $a_N$, we obtain (collecting terms proportional to $z$ in the second line)
\begin{equation}
  \label{eq:142MM}
  \begin{aligned}
    G_0^Nf(v,z) & + G_2^N h_N(v,z) - \varepsilon_f^N(v,z) \\ & = \big(- \kappa_{1} vz + \kappa_{2}' (M-z) - \kappa_1 v z a_N(v) + (\kappa_2' + \kappa_3')(M-z)a_N(v)\big)f'(v)\\ & = \big( M(\kappa_{2}' + a_N(v) (\kappa_{2}' + \kappa_{3}') ) - z\big(v\kappa_{1} + \kappa_{2}' + a_N(v) (\kappa_{1}v + \kappa_{2}' + \kappa_{3}')\big)f'(v).
\end{aligned}
\end{equation}
Choosing 
\begin{align}
  \label{eq:k1} a_N(v) = - \frac{v\kappa_{1} + \kappa_{2}'}{\kappa_{2}' + \kappa_{3}' + v\kappa_{1}} = - 1 + \frac{\kappa_{3}'}{\kappa_{2}' + \kappa_{3}' + v\kappa_{1}},
\end{align}
the last line of \eqref{eq:142MM} does not depend on $z$, and we have 
\begin{align*}
  (G_0^Nf + G_2^N h_N)(v,z) & - \varepsilon_f^N(v,z) = - M\Big(\kappa_{3}' - \frac{\kappa_{3}'(\kappa_{2}' + \kappa_{3}') }{\kappa_{2}' + \kappa_{3}' + v\kappa_{1}}\Big)f'(v)
 = p^N(v)f'(v).  
\end{align*}
With $Gf(v) = p(v) f'(v)$, we see that $\varepsilon_f^N(v,z) + (p(v) - p^N(v))f'(v)$ is as in \eqref{eq:defeps}. Since $p^N \xrightarrow{N\to\infty} p$, the compact containment condition for $(V^N)_{N=1,2,...}$ holds since $V^N(t)$ is bounded by $V^N(0) + M/N$, and \eqref{eq:T1a}--\eqref{eq:T1d} hold due to $0 \leq Z^N \leq M$ and \eqref{eq:fci}, we have shown that $V^N\xRightarrow{N\to\infty} V$.

~

\noindent
2. For the CLT, we write $U^N = N^{1/2}(V^N - W^N)$, and have for the Markov process $(U^N,Z^N,W^N)$ the generator for smooth $f$, only depending on $u,z$, writing $f'$ for the derivative with respect to $u$,
\begin{equation}
  \notag 
  \begin{aligned}
    L^Nf(u,z) & = (Nw + N^{1/2} u) \kappa_{1}    z(f(u-N^{-1/2},z-1) - f(u,z)) \\ & \qquad \qquad + \kappa_{2}'N(M-z) (f(u+N^{-1/2},z+1) - f(u,z)) \\  & \qquad \qquad + \kappa_{3}'  N(M-z) (f(u, z+1) - f(u,z)) + N^{1/2} M p^N(w) 
    f'(u,z) \\ & =
    L_{0}^N f(u,z) + N^{1/2} L_{1}^N f(u,z,w) + NL_{2}^N f(u,z), \\
    L_{0}^N f(u,z) & = \tfrac 12  \kappa_1 wz f''(u,z-1) + \tfrac  12 \kappa_{2}' (M-z) f''(u,z+1) - \kappa_{1} uz f'(u,z-1) + \varepsilon_f^N(u,z), \\
    L_{1}^Nf(u,z,w) & = - \kappa_{1} wz f'(u,z-1) + \kappa_{1} uz (f(u,z-1)- f(u,z)) \\ &  \qquad \qquad \qquad + \kappa_{2}' (M-z)  f'(u,z+1) + M p^N(w)    f'(u,z),\\
    L_{2}^N f(u,z) & = \kappa_{1} wz 
    (f(u,z-1) - f(u,z)) + (\kappa_{2}' + \kappa_{3}')(M-z) (f(u,z+1) - f(u,z)), \\ |\varepsilon_f^N(u,z)| & \leq N^{-1/2} \big(\tfrac 12  \kappa_1 u z ||f''||_\infty + \tfrac 16  (\kappa_1 w z + \kappa_2'(M-z)) ||f'''||_\infty).
  \end{aligned}
\end{equation}
Then, for some smooth $f$ depending only on $u$, taking $g_N(u,z,w) = z a_N(w) f'(u)$ with $a_N$ from \eqref{eq:k1} (but depending on $w$), we have $L_1^N f + L_2^Ng_N = 0$. So, we need to look for $h_N$ such that $L_0^N f + L_1^N g_N + L_2^N h_N$ only depends on $u, w$ and has a proper limit as $N\to\infty$. Choosing
\begin{align*}
    h_N(u,z,w) = b_N(u,w) z f'(u) + \Big(c_N(w) z + d_N(w)\binom z2\Big)f''(u),
\end{align*} 
for some $b_N, c_N, d_N$, we find, collecting terms proportional to $z$ and $z^2$ after the second equality, abbreviating $a:=a_N(w), b:=b_N(u,z,w),...$
\begin{align*}
    (L_{0}^N & f + L^N_{1}g_N + L_{2}^N h_N)(u,z,w) - \varepsilon_f^N(u,z,w)
    \\ & =
    \Big(- \kappa_1 uz - \kappa_1 u z a -\kappa_1 zw b + (\kappa_2' + \kappa_3')(M-z) b  \Big) f'(u) \\ & \qquad + \Big(\tfrac 12 \kappa_1 wz + \tfrac 12 \kappa_2'(M-z) - \kappa_1 w z(z-1)a + \kappa_2'(M-z) a(z+1) + Mz p^N(w)a \\ & \qquad \qquad \qquad \qquad 
    - \kappa_1 z w (c + d(z-1)) + (\kappa_2' + \kappa_3')(M-z) (c + d z)   \Big) f''(u)
    \\ & = \Big(M (\kappa_2' + \kappa_3')b - z\Big(\kappa_1 u(1+a) + b(\kappa_1 w + \kappa_2' + \kappa_3'\Big) \Big) f'(u)
    \\ & \qquad + \Big(M\Big( \kappa_2' (\tfrac 12 + a) + (\kappa_2' + \kappa_3')c  \Big) \\ & \qquad + z\Big( \tfrac 12(\kappa_1 w - \kappa_2') + (\kappa_1 w + (M-1)\kappa_2' + Mp^N(w))a \\ & \qquad \qquad \qquad \qquad \qquad \qquad \qquad \qquad - (\kappa_1 w + \kappa_2' + \kappa_3')c + (\kappa_1 w + M(\kappa_2' + \kappa_3'))d\Big)
    \\ & \qquad \qquad \qquad \qquad \qquad \qquad \qquad \qquad  -z^2\Big((\kappa_1 w + \kappa_2') a + (\kappa_1 w + \kappa_2' + \kappa_3')d\Big)\Big) f''(u).
\end{align*}
Choosing $b, c, d$ such that the right hand side does not depend on $z$, we arrive after some rearrangements at
\begin{equation}
    \notag 
    \begin{aligned}
    ( L_0^N & f + L_1^N g_N + L_2^N h_N)(u,w)  = Lf(u,w) + \varepsilon_f^N(u,z,w) \\ Lf(u,w) & =  -\frac{M\kappa_1\kappa_3'(\kappa_2' + \kappa_3') }{(\kappa_1 w + \kappa_2' + \kappa_3')^2} u f'(u) +\Big( \frac{M\kappa_1 w \kappa_3' (\kappa_1 w + \kappa_2')^2}{2(\kappa_1 w + \kappa_2' + \kappa_3')^3} + \frac{M(\kappa_3')^2 \kappa_1 w(2\kappa_2' + \kappa_3')}{2(\kappa_1 w + \kappa_2'+ \kappa_3')^3} \Big) f''(u),
\end{aligned}
\end{equation}
where $Lf$ is the generator of $U$. Now, \eqref{eq:T1a}--\eqref{eq:T1d} follow from \eqref{eq:fci} and $0\leq Z^N_t \leq M$ and we are left with showing the compact containment condition for $(U^N)_{N=1,2,...}$. For this, let $g_N(t) := g_N(Z^N_t, W^N_t) = Z^N_t a_N(W^N_t)$ and $h(t) := h_N(U^N_t, Z^N_t, W^N_t) = Z^N_t b_N(U^N_t, W^N_t)$ (i.e.\ as above with $f(u) = u$). Then,  
\begin{align*}
U^N_t & = M^N_t - N^{-1/2}g_N(t) - N^{-1}h_N(t) - \int_0^t \tfrac{M\kappa_1\kappa'_3(\kappa'_2+\kappa'_3)}{(\kappa_1W^N_s+\kappa'_2+\kappa'_3)^2}U^N_s ds,
\\ M^N_t&:=U^N_t+N^{-1/2}g_N(t) + N^{-1} h_N(t)+\int_0^t \tfrac{M\kappa_1\kappa'_3(\kappa'_2+\kappa'_3)}{(\kappa_1 W^N_s+\kappa'_2+\kappa'_3)^2}U^N_s ds.
\end{align*}
Here, $(M^N_t)_{t\geq 0}$ is a martingale with quadratic variation
\begin{align*}
    [M^N]_t&=N[V^N+N^{-1}g_N]_t\\
    & \approx \int_0^t (1 + a_N(W_s^N))^2(\kappa_1 V^N_sZ^N_s + \kappa_2'  (M-Z^N_s))  +  a_N(W^N_s)^2 \kappa'_3 (M-Z^N_s)ds.
\end{align*}
Since $V^N,Z^N,W^N$ are bounded, $g_N$ and $h_N$ are bounded as well. Moreover, from \eqref{eq:fci} we see that the quadratic variation of $M^N$ is locally bounded, so we can conclude that there are $a_S$ and $b_S$ such that \eqref{eq:gron} holds, and the compact containment condition for $(U^N)_{N=1,2,...}$ follows. 
\end{proof}

\subsection{Hill coefficient dynamics}
\label{ss:Hill}
We now pick up Example~\ref{ex:simHill1}, but extend the CRN in order to capture all required reactions for binding of two ligands. This also generalizes the Michaelis-Menten dynamics, where the macromolecule (enzyme) can only be bound by a single ligand. Specifically, we study
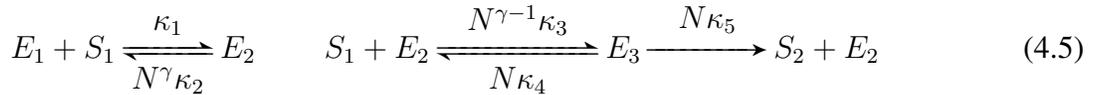
\begin{align}\label{eq:CRNhill}
  \operatorname{\text{$E_1+S_1$}} \autorightleftharpoons{\text{$\kappa_{1}$}}{\text{$N^\gamma \kappa_{2}$}}   \operatorname{\text{$E_2$}} \qquad \operatorname{\text{$S_1 + E_2$}} \autorightleftharpoons{\text{$N^{\gamma-1}\kappa_{3}$}}{\text{$N\kappa_4$}} \operatorname{\text{$E_3$}} \autorightarrow{\text{$N\kappa_5$}}{}    \operatorname{\text{$S_2+E_2$}}
\end{align}
for $\gamma>1$. Here, similar to Michaelis-Menten kinetics, $E_1$ (the macromolecule) and $S_1$ (the ligand) form a short-lived complex $E_2$, but only after forming another complex $E_3$ together with another $S_1$, the product $S_2$ is formed and $E_2$ is released. We will assume that $S_1, S_2$ are in high abundance (order $N$) and $E_1, E_2, E_3$ are in low abundance (order 1). Since $\gamma>1$, the balance condition fails for $E_1, E_2$ and $E_3$. Let $NV_1^N, NV_2^N$ and $Z_1,Z_2,Z_3$ be the processes of the particle numbers of $S_1,S_2$ and $E_1, E_2, E_3$, respectively. Since $M:=Z_1+Z_2+Z_3$ is a constant, and $V_1 + V_2 + N^{-1} (Z_1 + 2Z_2 + 3Z_3)$ is constant as well, we are going to describe the Markov process $(Z^N = (Z^N_1, Z^N_3), V^N=V_1^N)$.
  
\begin{proposition}[LLN and CLT for \eqref{eq:CRNhill}]
  Let $u\in\mathbb R$ and $v\in\mathbb R_+$, 
  $$ p^N(v) := -\frac{\kappa_1N^{\gamma-1}\kappa_3\kappa_5 v^2}{(\kappa_1 v + N^{\gamma-1}\kappa_2)(\kappa_4 + \kappa_5) + N^{\gamma-1}\kappa_1\kappa_3v^2} \xrightarrow{N\to\infty} p(v):= -\frac{\kappa_1\kappa_3\kappa_5 v^2}{\kappa_2(\kappa_4 + \kappa_5) + \kappa_1\kappa_3v^2}$$
  and $V$ and $W^N$ the solutions of the ODEs
  \begin{align*}
      dV = p(V) dt, \qquad dW^N = p^N(W^N) dt, \qquad  V_0=W^N_0=v,
  \end{align*}
  Moreover, let $U$ be the solution of the SDE
  \begin{align*}
    dU &= -2M \frac{\kappa_1\kappa_2\kappa_3(\kappa_4 + \kappa_5)\kappa_5uV}{(\kappa_2(\kappa_4 + \kappa_5) + \kappa_1\kappa_3 V^2 )^2} dt
    \\ & \qquad + \sqrt{\frac{(\kappa_2^2(\kappa_4+\kappa_5)^2 + 2\kappa_2\kappa_3\kappa_5(\kappa_4+\kappa_5)V + (\kappa_1\kappa_3V^2 + 2\kappa_2\kappa_4)\kappa_1\kappa_3V^2)M\kappa_1\kappa_3\kappa_5V^2}{(\kappa_2(\kappa_4 + \kappa_5) + \kappa_1\kappa_3V^2)^3}} dB
  \end{align*}  
  for some Brownian motion $B$. Then, let $V^N$ as above, and $U^N = N^{1/2}(V^N - W^N)$. If $V^N_0 \xRightarrow{N\to\infty} v$ and $U^N_0 \xRightarrow{N\to\infty} u$, then
  $$ V^N \xRightarrow{N\to\infty} V, \qquad U^N \xRightarrow{N\to\infty} U. $$
\end{proposition}
  
\begin{remark}
\begin{enumerate}
  \item Instead of the rates $N\kappa_4, N\kappa_5$, we could use a scaling $N^\beta \kappa_4, N^\beta\kappa_5$ for $\beta\geq 1$. This would change the result, but not the proof, except for taking the limits in the last equalities in \eqref{eq:hill_limit_LLN} for the LLN and in \eqref{eq:hill_lim} for the CLT.
  \item The same holds for the case $\gamma=1$: In fact, $\gamma>1$ only plays a role in taking the last limits in \eqref{eq:hill_limit_LLN} and \eqref{eq:hill_lim}.
  \item In the CLT, if $\gamma>3/2$, we find
  \begin{align*}
      N^{1/2}(p^N(v) - p(v))& \xrightarrow{N\to\infty} 0,  
  \end{align*}
  and -- by an application of Gronwalls lemma -- $N^{1/2}(\tilde V^N - V) \xRightarrow{N\to\infty} 0$. Therefore, in this case, we find in addition that $N^{1/2}(V^N - V) \xRightarrow{N\to\infty} U$.
\end{enumerate}
\end{remark}  
  
\begin{proof}
We set  
\begin{align*} 
  \kappa'_1:=N^{\gamma-1}\kappa_2, \qquad \kappa'_2=N^{\gamma-1}\kappa_3.
\end{align*}
The generator of $(V^N,Z^N)$ is for smooth $f$, writing $f'$ for the derivative according to the first variable,
\begin{align*}
  G^Nf(v,z)&=N\kappa_1vz_1(f(v-N^{-1},z-e_1)-f(v,z))\\ &\qquad+N \kappa'_1 (M-z_1-z_3)(f(v+N^{-1},z+e_1)-f(v,z))\\
  &\qquad +N\kappa'_2v (M-z_1-z_3) (f(v-N^{-1},z+e_3)-f(v,z))\\
  &\qquad +N\kappa_4 z_3 (f(v+N^{-1},z-e_3)-f(v,z))+N\kappa_5 z_3 (f(v,z-e_3)-f(v,z))\\ 
  &=G_0^Nf(v,z)+NG_1^Nf(v,z) + \varepsilon^N + (M-z_1-z_3) O(N^{\gamma-2}) + z_3O(N^{\beta-2})+o(1),\\
  G_0^Nf(v,z)&=-\kappa_1 v z_1  f'(v,z-e_1)+\kappa'_1(M-z_1-z_3) f'(v,z+e_1)\\
  &\qquad -\kappa'_2 v(M-z_1-z_3) f'(v,z+e_3)+\kappa_4 z_3 f'(v,z-e_3),\\
  G_1^Nf(v,z)&=\kappa_1 v z_1  (f(v,z-e_1)-f(v,z))+ \kappa'_1(M-z_1-z_3)(f(v,z+e_1)-f(v,z))\\
  &\qquad +\kappa'_2 v (M-z_1-z_3)(f(v,z+e_3)-f(v,z))+\kappa_4 z_3(f(v,z-e_3)-f(v,z))\\
  &\qquad +\kappa_5 z_3(f(v,z-e_3)-f(v,z)),\\
  \varepsilon_1^N&=N^{-1} \big(\tfrac 1 2\kappa_1 v z_1   f''(v,z-e_1)+\tfrac 1 2\kappa'_1(M-z_1-z_3)f''(v,z+e_1)\\
  &\qquad +\tfrac 1 2 \kappa'_2 v(M-z_1-z_3) f''(v,z+e_3)+\tfrac 1 2\kappa_4 z_3 f''(v,z-e_3)\big).
\end{align*}

~

\noindent
1. For the LLN, take some smooth $f$ depending only on $v$. We need to find $g$ such that the limit of $G_0^N f + G_1^N g $  exists and only depends on $v$. Choosing $$g(v,z)=(z_1a_1+z_3a_3)f'(v)$$ 
for some functions $a_1, a_3$, we find
\begin{equation}
\label{eq:moi}
  \begin{aligned}
    (G_0^Nf + G_1^Ng)(v,z)&=\Big(M \Big(\kappa'_1-\kappa'_2v+\kappa'_1a_1+\kappa'_2 va_3\Big)\\  &\qquad +z_1 \Big(-\kappa_1v - \kappa'_1+\kappa'_2 v-(\kappa_1 v + \kappa_2')a_1 - \kappa'_2 va_3 \Big)\\ & \qquad + z_3\Big(-\kappa'_1+\kappa'_2 v+\kappa_4 - \kappa'_1 a_1 - (\kappa'_2 v + \kappa_4 + \kappa_5) a_3\Big)\Big)f'(v).
  \end{aligned}
\end{equation}
With
\begin{align*}
  a_1 & :=-\frac{(\kappa_1\kappa'_2v^2+\kappa'_1(\kappa_4+\kappa_5)+(\kappa_1(\kappa_4+\kappa_5)-\kappa'_2\kappa_5)v}{\kappa_1\kappa'_2v^2+\kappa_1(\kappa_4+\kappa_5)v+\kappa'_1(\kappa_4+\kappa_5)},\\ 
  a_3 & :=\frac{\kappa_1\kappa'_1v^2+\kappa_1\kappa_4v+\kappa'_1\kappa_4}{\kappa_1\kappa'_2v^2+\kappa_1(\kappa_4+\kappa_5)v+\kappa'_1(\kappa_4+\kappa_5)},
\end{align*}
the last two lines in \eqref{eq:moi} vanish and we obtain the generator of the limit
\begin{equation}\label{eq:hill_limit_LLN}
\begin{aligned}
    Gf(v) & = \lim_{N\to\infty} (G_0^N f+ G_1^N g)(v) = M \lim_{N\to\infty} (\kappa'_1-\kappa'_2v+\kappa'_1a_1+\kappa'_2 va_3)f'(v) \\ & = -M \lim_{N\to\infty}
    \frac{\kappa_1 \kappa_3' \kappa_5 v^2}{\kappa_2'(\kappa_4 + \kappa_5) + \kappa_1(\kappa_4 + \kappa_5) v + \kappa_1\kappa_3'v^2}f'(v) \\ & = -M \frac{\kappa_1 \kappa_3 \kappa_5 v^2}{\kappa_2(\kappa_4 + \kappa_5) + \kappa_1\kappa_3v^2}f'(v).
\end{aligned}
\end{equation}
~

\noindent
2. For the fluctuations, consider $U = N^{1/2}(V^N-\tilde V^N)$ and the generator of the Markov process $(U,Z)$ is
\begin{align*}
  L^Nf(u,z)&=\kappa_1z_1(Nv+N^{1/2}u)(f(u-N^{-1/2},z-e_1)-f(u,z))\\ & \qquad+N \kappa'_1 (M-z_1-z_3)(f(u+N^{-1/2},z+e_1)-f(u,z))\\  &\qquad +\kappa'_2 (Nv+N^{1/2}u) (M-z_1-z_3) (f(u-N^{-1/2},z+e_3)-f(u,z))\\ &\qquad +N\kappa_4 z_3 (f(u+N^{-1/2},z-e_3)-f(u,z)) + N\kappa_5 z_3 (f(u,z-e_3)-f(u,z))\\  &\qquad -N^{1/2}p^N(v) f'(u,z)\\      
  &=L_0^Nf(u,z)+N^{1/2}L_1^N f(u,z)+NL_2^Nf(u,z) + o(1),\\      L_0^Nf(u,z)&=-\kappa_1 u z_1   f'(u,z-e_1)+\tfrac 12 \kappa_1 v z_1 f''(u,z-e_1)\\
  &\qquad +\tfrac 12 \kappa'_1(M-z_1-z_3) f''(u,z+e_1)-\kappa'_2u(M-z_1-z_3) f'(u,z+e_3)\\
  &\qquad +\tfrac 12 \kappa'_2 v(M-z_1-z_3)  f''(u,z+e_3) +\tfrac 12 \kappa_4 z_3 f''(u,z-e_3)\\
  L_1^Nf(u,z)&=-\kappa_1vz_1  f'(u,z-e_1)+\kappa_1 uz_1(f(u,z-e_1)-f(u,z)) \\  &\qquad +\kappa'_1(M-z_1-z_3) f'(u,z+e_1)-\kappa'_2v(M-z_1-z_3) f'(u,z+e_3)\\
  &\qquad +\kappa'_2u(M-z_1-z_3)(f(u,z+e_3)-f(u,z_3))+\kappa_4z_3  f'(u,z-e_3)-p^N(v) f'(u,z)\\
  L_2^Nf(u,z)&=\kappa_1 vz_1  (f(u,z-e_1)-f(u,z))+ \kappa'_1(M-z_1-z_3)(f(u,z+e_1)-f(u,z))\\  &\qquad +\kappa'_2v (M-z_1-z_3)(f(u,z+e_3)-f(u,z))+\kappa_4z_3(f(u,z-e_3)-f(u,z))\\   &\qquad +\kappa_5 z_3(f(u,z-e_3)-f(u,z)).
  \end{align*}
With $f$ only depending on $u$ and $g(u,z) = (z_1 a_1 + z_3a_3)f'(u)$ as above but depending on $u$ instead of $v$, we have that $L_1^N f+L_2^N g=0$. Then, we make the ansatz
$$h(u,z)=\Big(z_1b_1+z_3b_3\Big)f'(u) + \Big(z_1c_1+z_3c_2+\binom{z_1}{2}d_1+\binom{z_3}{2}d_2+\binom{z_1+z_3}{2}d_3\Big)f''(u)$$
for some $b_1, b_3, c_1,\dots,d_3$ and obtain
\begin{align*}
  (L_0^N & f+L_1^Ng+L_2^Nh)(u,z)\\
  &=\Big(M\big(\kappa_3'u(a_3-1)+\kappa_2'b_1+\kappa_3'vb_3\big)\\  &\qquad+z_1\big(-\kappa_1u+\kappa_3'u-\kappa_1ua_1-\kappa_3'ua_3-\kappa_1vb_1-\kappa_2'b_1-\kappa_3'vb_3\big)\\ &\qquad + z_3\big(\kappa_3'u-\kappa_3'a_3u-\kappa_2'b_1-\kappa_3'vb_3-\kappa_4b_3-\kappa_5b_3\big)\Big)f'(u)\\
  &+\Big(M\big(\tfrac 12(\kappa_2'+\kappa_3'v)+\kappa_2'a_1-\kappa_3'va_3+\kappa_2'c_1+\kappa_3'vc_2\big) \\ &\qquad+z_1\big(\tfrac 12(\kappa_1v-\kappa_2'-\kappa_3'v)+\kappa_1va_1+\kappa_2Ma_1-\kappa_2a_1-\kappa_3Mva_1-\nu a_1-\kappa_1vc_1\\  &\quad\qquad +\kappa_1vd_1+\kappa_1vd_3-\kappa_2'c_1+\kappa_2'Md_1+\kappa_2'Md_3-\kappa_3'vc_2+\kappa_3'Mvd_3+\kappa_3'vb\big)\\ &\qquad+z_3\big(\frac 12(-\kappa_2'-\kappa_3'+\kappa_4)+\kappa_2'Mb-\kappa_2'a-\kappa_3'Mvb+\kappa_3'vb\\   &\quad\qquad-\kappa_4b-\nu b-\kappa_2'c_1+\kappa_2'Md_3-\kappa_3'vc_2+\kappa_3'Mvd_2+\kappa_3'Mvd_3\\ &\quad\qquad -\kappa_4c_2+\kappa_4d_2+\kappa_4d_3-\kappa_5c_2+\kappa_5d_2+\kappa_5d_3\big)\\   &\qquad+z_1^2\big(-\kappa_1va_1-\kappa_2'a_1+\kappa_3'va_1-\kappa_1vd_1-\kappa_1vd_3-\kappa_2'd_1-\kappa_2'd_3-\kappa_3'vd_3\big)\\ &\qquad+z_3^2\big(-\kappa_2'a_3+\kappa_3'va_3+\kappa_4b-\kappa_2'd_3-\kappa_3'vd_2-\kappa_3'vd_3-(\kappa_4+\kappa_5)(d_2+d_3)\big)\\ &\qquad+z_1z_3\big(-\kappa_1va_3-(\kappa_2'-\kappa_3'v)(a_1+a_3)+\kappa_4a_1-\kappa_1vd_3-\kappa_2'(d_1+2d_3)\\  &\quad\qquad -\kappa_3 v(d_2+2d_3)-\kappa_4d_3-\kappa_5d_3\big)\Big)f''(u).
\end{align*}
We can choose $b_1, b_3, c_1,...,c_5$ such that the right hand side only depends on $v$. Plugging these into the last display gives, in the limit $N\to\infty$
\begin{equation}\label{eq:hill_lim}
 \begin{aligned}
  Lf(u)&=\lim_{N\to\infty} (L_0^Nf+L_1^Ng+L_2^Nh)(u,z))\\ & = \lim_{N\to\infty} M\big(\kappa_3'u(a_3-1)+\kappa_2'b_1+\kappa_3'vb_3\big)\cdot f'(u)\\ &\qquad\qquad+M\big(\tfrac 12(\kappa_2'+\kappa_3'v)+\kappa_2'a_1-\kappa_3'va_3+\kappa_2'c_1+\kappa_3'vc_2\big)\cdot f''(u) \\ &  = -\frac{2M \kappa_1 \kappa_2 \kappa_3(\kappa_4 + \kappa_5)\kappa_5 u v}{(\kappa_1 \kappa_3 v^2 + \kappa_2(\kappa_4 + \kappa_5)^2} f'(u) \\ & \quad 
  + \frac 12 \frac{(\kappa_2^2(\kappa_4+\kappa_5)^2 + 2\kappa_2\kappa_3\kappa_5(\kappa_4+\kappa_5)v + (\kappa_1\kappa_3v^2 + 2\kappa_2\kappa_4)\kappa_1\kappa_3v^2)M\kappa_1\kappa_3\kappa_5v^2}{(\kappa_1\kappa_3v^2 + \kappa_2(\kappa_4 + \kappa_5))^3} f''(u). 
  \end{aligned}
\end{equation}
Conditions~\eqref{eq:T1a}--\eqref{eq:T1d} as well the compact containment condition for $(U^N)_{N=1,2,...}$ can be shown as in Example~\ref{ex:simHill5}. 

\end{proof}

\subsection{Extending the first example from Section 6.5 of~[11]} 
\label{ss:ABCD}
Here, we extend the first example in Section 6.5 of \cite{KangKurtz2013}. Precisely, we study
\begin{align}
\label{M:ABCD}
  \operatorname{\text{$S_1$}} \autorightleftharpoons{\text{$\kappa_{1}$} }{\text{$N^\gamma \kappa_2$}}
  \operatorname{\text{$E_1$}} \autorightleftharpoons{\text{$N^\gamma\kappa_{3}$}}{ \text{$N^\beta\kappa_4$}}
  \operatorname{\text{$E_2$}}\autorightarrow{ \text{$N^\beta\kappa_{5}$}}{} \operatorname{\text{$S_2$}}.
\end{align}
for $\beta, \gamma \geq 1$.  Here, $S_i$ is in high abundance (order $N$), and $E_i$ is in low abundance (order 1), $i=1,2$, so we set $Z_i$ as the copy number of $E_i$ and $NV_i$ as the copy numbers of species $S_i$, $i=1,2$. Since $V_1^N + V_2^N + (Z_1^N + Z_2^N)N^{-1}$ is constant, it suffices to study $(V^N=V_1^N, Z^N = (Z_1^N, Z_2^N))$. 

Looking at Assumption~\ref{ass:1}, we note that $\widetilde\beta_0 = 1, \widetilde \beta_1 = \widetilde \beta_2 = \gamma, \widetilde \beta_3 = \widetilde \beta_4 = \beta$, as well as $\varphi_{E_1} = \gamma, \varphi_{E_2} = \beta$, leading to $\widehat\beta_1 = \widehat\beta_2 = \widehat\beta_2 + 1 - \varphi_{E_1} = 1, \widehat\beta_3 = \widehat\beta_4 = \widehat\beta_4 + 1 - \varphi_{E_2} = 1$ and $\psi_{E_1} = \psi_{E_2} = 1$ and \eqref{eq:prodcun2} is satisfied.

\begin{proposition}[LLN and CLT for \eqref{M:ABCD}]
  Let $u\in\mathbb R$ and $v\in\mathbb R_+$, $V$ the solution of the ODE
  $$ dV = - \frac{\kappa_1 \kappa_3 \kappa_5 v}{\kappa_2(\kappa_4 + \kappa_5) + \kappa_3 \kappa_5} dt, \qquad V_0=v,$$
  and $U$ the solution of the SDE
  $$ dU = - \frac{\kappa_1 \kappa_3 \kappa_5 u}{\kappa_2(\kappa_4 + \kappa_5) + \kappa_3 \kappa_5} dt + \sqrt{\frac{\kappa_1 \kappa_3 \kappa_5 v}{\kappa_2(\kappa_4 + \kappa_5) + \kappa_3 \kappa_5}} dB, \qquad U_0=u,$$
  for some Brownian motion $B$. Then, letting $V^N$ be as above, setting $U^N = N^{1/2}(V^N - V)$, and if $V^N_0\xRightarrow{N\to\infty} v$ and $U^N_0\xRightarrow{N\to\infty} u$, we have
  $$ V^N\xRightarrow{N\to\infty} V, \qquad U^N\xRightarrow{N\to\infty} U.$$
\end{proposition}

\begin{proof}
We define
\begin{align}\label{eq:auv}
 \kappa'_1:=N^{\gamma-1}\kappa_2, \qquad \kappa'_2:=N^{\gamma-1}\kappa_3, \qquad 
 \kappa'_3:=N^{\beta-1} \kappa_4, \qquad
 \kappa'_4:=N^{\beta-1} \kappa_5.
\end{align}
Then, $(V^N, Z^N)$ has the generator for $f \in \mathcal C_c^\infty(\mathbb R \times \mathbb R^2)$, where $f'$ is the derivative with respect to the first variable,
\begin{equation}
  \notag 
  \begin{aligned}
    G^Nf(v,z) & = Nv \kappa_1 (f(v - N^{-1}, z + e_1) - f(v,z)) \\ & \qquad + N \kappa'_{1} z_1 (f(v + N^{-1}, z-e_1) - f(v,z)) + N z_1 \kappa'_{2} (f(v,z + e_2 - e_1) - f(v,z)) \\ & \qquad + N\kappa'_3 z_2 (f(v,z+e_1 - e_2) - f(v,z)) + N\kappa'_4 z_2(f(v, z-e_2) - f(v,z)) \\ & =  G_0^Nf(v,z) + N G_2^Nf(v,z),\\ G_0^Nf(v,z) & = - \kappa_1 v f'(v,z+e_1) + z_1 \kappa'_1 f'(v,z-e_1) + \varepsilon^N_f(v,z), \\ G_2^Nf(v,z) & = \kappa_1v(f(v,z+e_1)-f(v,z))+\kappa'_{1} z_1(f(v,z-e_1) - f(v,z))\\
    &\qquad +\kappa'_2 z_1 (f(v,z+e_2 - e_1) - f(v,z)) + \kappa'_3 z_2 (f(v,z+e_1 - e_2) - f(v,z))\\
    &\qquad + \kappa'_4 z_2(f(v, z-e_2) - f(v,z)),\\
    |\varepsilon_f^N(v,z)|& \leq N^{-1}\tfrac 12 (\kappa_1 v + \kappa'_1 z_1)||f''||_\infty.
  \end{aligned}
\end{equation}
For bounding the two fast variables, a calculation similar to the proof of Lemma~\ref{l:ccc} gives, for all $T,a,k>0$
\begin{equation}
    \label{eq:eq9}
    \begin{aligned}
      & \mathbb E\Big[ \int_0^T (Z_1^N(t))^k dt\Big] = O(N^{1-\gamma}), \qquad     \mathbb E\Big[ \int_0^T (Z_2^N(t))^k dt\Big] = O(N^{1-\beta}), 
      \\ & N^{-a} \mathbb E[\sup_{0\leq t\leq T} (Z_1^N(t))^k] \xrightarrow{N\to \infty}0, \qquad N^{-a} \mathbb E[\sup_{0\leq t\leq T} (Z_2^N(t))^k] \xrightarrow{N\to \infty}0.
    \end{aligned}
\end{equation}
1. Again, $G_1^N=0$, and we take $g_N=0$. The compact containment condition for $(V^N)_{N=1,2,...}$ is straight-forward, since $V^N + (Z_1^N + Z_2^N)/N)$ is non-increasing and \eqref{eq:eq9} bounds $(Z_1^N)_{N=1,2,...}$ and $(Z_2^N)_{N=1,2,...}$. Using the same arguments, conditions \eqref{eq:T1a}--\eqref{eq:T1d} will hold, such that we can concentrate on generator calculations:\\
Take $f \in \mathcal C_c^\infty(\mathbb R)$ only depending on $v$. We are looking for $h_N$ such that $G_0^Nf(v,z) + G_2^Nh_N(v,z)$ only depends on $N, v$ and has a limit for $N\to\infty$. Choosing (see \eqref{eq:g1})
$$g(v,z) = (a_1 z_1 + a_2 z_2) f'(v)$$ for some $a_1,a_2$ (which might depend on $N$), we find
\begin{align*}
    G_0^N & f(v,z)+ G_1^N h_N(v,z) =-\kappa_1vf'(v)+z_1\kappa'_1f'(v)+\kappa_1va_1f'(v)-\kappa'_1z_1a_1f'(v)\\
    &\qquad\qquad\qquad +\kappa'_2z_1(a_2-a_1)f'(v)+\kappa'_3z_2(a_1-a_2)f'(v)-\kappa'_4z_2a_2f'(v) + \varepsilon_f^N(v,z)\\
    &=\big( \kappa_1v(a_1-1)+z_1(\kappa'_1 - (\kappa_2' + \kappa_3')a_1 + \kappa_3' a_2) +z_2(\kappa'_3 a_1 - (\kappa_4' + \kappa'_4)a_2\big)f'(v)  + \varepsilon_f^N(v,z).
\end{align*}
Choosing $a_1, a_2$ such that the terms proportioal to $z_1$ and $z_2$ vanish, i.e.\,
\begin{align*}
a_1 & =\frac{\kappa'_1(\kappa'_3+\kappa'_4)}{\kappa'_1(\kappa'_3+\kappa'_4) + \kappa'_2\kappa'_4} =\frac{\kappa_2(\kappa_4+\kappa_5)}{\kappa_2(\kappa_4+\kappa_5) + \kappa_3\kappa_5}, \\ a_2&=\frac{\kappa'_1\kappa'_3}{\kappa'_1(\kappa'_3+\kappa'_4)  +\kappa'_2\kappa'_4}=\frac{\kappa_2\kappa_4}{\kappa_2(\kappa_4+\kappa_5)  +\kappa_3\kappa_5},
\end{align*}
where we have used \eqref{eq:auv}, we obtain
\begin{align*}
  (G_0^N f + G_1^N h_N)(v,z) & = \kappa_1 v (a_1-1)f'(v) + \varepsilon_f^N(v,z) = Gf(v) + \varepsilon_f^N(v,z),
  \\ \text{ with }Gf(v) & = -\frac{\kappa_1\kappa_3\kappa_5v}{\kappa_2(\kappa_4+\kappa_5)+\kappa_3\kappa_5}f'(v).
\end{align*}
Since $G$ is the generator of $V$, we are done.

~

\noindent
2. For the functional CLT, we start with generator calculations. Writing $U = N^{1/2}(V^N - V)$ and have for the generator of $(U^N, Z^N,V)$ (again, $Z^N = (Z^N_1, Z^N_2)$), using $v$ from the LLN, for $f \in \mathcal C_c^\infty(\mathbb R \times \mathbb R^2)$, only depending on $u,z$, with $a_1$ from above,
\begin{equation}
  \notag 
  \begin{aligned}
    L^Nf(u,z,v) & = (Nv + N^{1/2} u) \kappa_1 (f(u - N^{-1/2}, z + e_1) - f(u,z)) \\ & \qquad + N \kappa'_{1} z_1 (f(u + N^{-1/2}, z-e_1) - f(u,z)) + N \kappa'_{2} (f(u,z + e_2 - e_1) - f(u,z))  \\ & \qquad + N\kappa'_3 z_2 (f(u,z+e_1 - e_2) - f(u,z)) + N\kappa'_4 z_2(f(u, z-e_2) - f(u,z)) \\ & \qquad \qquad \qquad\qquad \qquad \qquad  \qquad \qquad \qquad - N^{1/2} \kappa_1 v(a_1-1) f'(u,z) \\ & =  L_0^Nf(u,z,v) + N^{1/2} L_1^Nf(u,z,v) + N L_2^Nf(u,z,v), \\
    L_0^Nf(u,z,v) & = \tfrac 12 \kappa_1 v f''(u,z+e_1)-\kappa_1 u f'(u,z+e_1)+\tfrac 12 \kappa'_1 z_1f''(u,z-e_1) + \varepsilon_f^N(u,z,v),\\ 
    L_1^Nf(u,z,v) & = -\kappa_1 v f'(u,z+e_1)+\kappa_1 u (f(u,z+e_1)-f(u,z))+\kappa'_1z_1f'(u,z-e_1) \\ & \qquad \qquad \qquad\qquad \qquad\qquad \qquad \qquad  \qquad \qquad \qquad - \kappa_1 v(a_1-1)f'(u,z),\\
    L_2^Nf(u,z) & = \kappa_1v(f(u,z+e_1)-f(u,z))+\kappa'_1z_1(f(u,z-e_1)-f(u,z))\\
    &\qquad +\kappa'_2z_1(f(u,z-e_1+e_2)-f(u,z))+\kappa'_3z_2(f(u,z+e_1-e_2)-f(u,z))\\
    &\qquad +\kappa'_4z_2(f(u,z-e_2)-f(u,z)),
    \\ |\varepsilon_f^N(u,z,v)| & \leq \tfrac 16 N^{-1/2} (\kappa_1 v + \kappa_2' z_1)||f'''||_\infty.
  \end{aligned}
\end{equation}
With $g(u,z) = (a_1 z_1 + a_2 z_2)f'(u)$ as above, $L_1^N f + L_2^N g_N = 0$. 
Next, let us look for $h_N$ such that $L_0^Nf + L_1^Ng_N + L_2^Nh_N$ doesn't depend on $z$ and has a limit for $N\to\infty$. With (recall from \eqref{eq:hNgen}, where we set $c=0$)
$$h_N(u,z)=\Big(z_1b_1+z_2b_2+\binom{z_1}{2}d_1+\binom{z_2}{2}d_2+\binom{z_1+z_2}{2}d_3\Big)f''(u)$$
for some $b_1,b_2,d_1,d_2,d_3$, we have
\begin{align*}
&L_0^Nf(u,z,v)+L_1^Ng_N(u,z,v)+L_2^Nh_N(u,z,v) - \varepsilon_f^N(u,z,v)\\
&\qquad =\tfrac 12 \kappa_1 v f''(u)-\kappa_1 u f'(u)+\tfrac 12 \kappa'_1z_1f''(u)-\kappa_1v(a_1(z_1+1)+a_2z_2)f''(u)\\
&\qquad\qquad +\kappa_1ua_1f'(u)+\kappa'_1z_1(a_1(z_1-1)+a_2z_2)f''(u) - \kappa_1 v(a_1-1) (a_1 z_1+a_2z_2)f''(u)\\
&\qquad\qquad +\kappa_1 v(b_1+z_1d_1+(z_1+z_2)d_3)f''(u)\\
&\qquad\qquad -\kappa'_1z_1(b_1+(z_1-1)d_1+(z_1+z_2-1)d_3)f''(u)\\
&\qquad\qquad +\kappa'_2z_1(-b_1+b_2-(z_1-1)d_1+z_2d_2)f''(u)\\
&\qquad\qquad +\kappa'_3z_2(b_1-b_2+z_1d_1-(z_2-1)d_2)f''(u)\\
&\qquad\qquad -\kappa'_4z_2(b_2+(z_2-1)d_2+(z_1+z_2-1)d_3)f''(u)\\
&\qquad = \kappa_1 u (a_1-1) f'(u) + \Big(\tfrac 12 \kappa_1 v(1 - 2a_1 + 2b_1) \\
&\qquad\qquad +z_1\big(\tfrac 12 \kappa_2' - v\kappa_1 a_1^2 -(\kappa_2' + \kappa_3') b_1 + \kappa_3' b_2 + v\kappa_1 d_1 +  v\kappa_1 d_3\big)\\
&\qquad\qquad +z_2\big(-v\kappa_1 a_1 a_2 + \kappa_4' b_1 - (\kappa_4' + \kappa_5') b_2 + v\kappa_1 d_3\big)\\
&\qquad\qquad +z_1(z_1-1)\big(\kappa_2' a_1 - (\kappa_2' + \kappa_3')d_1 - \kappa_2' d_3\big)\\
&\qquad\qquad +z_2(z_2-1)\big(-(\kappa'_3 + \kappa_5')d_2-\kappa'_4 d_3\big)\\
&\qquad\qquad \Big.+z_1z_2\big(\kappa'_1a_2+\kappa'_3d_1+\kappa_3'd_2-(\kappa'_1 + \kappa_5')d_3\big)\Big)f''(u).
\end{align*}
Now, we need to find $b_1,b_2,d_1,d_2,d_3$ such that the last five lines vanish. 
This leads to 
$$ b_1 = \tfrac 12 \cdot \frac{\kappa_2' (\kappa_4' + \kappa_5')}{\kappa_2'(\kappa_4' + \kappa_5') + \kappa_3' \kappa_5'}= \tfrac 12 a_1.$$
For the limiting generator, we thus have, using \eqref{eq:auv},
\begin{align*}
(L_0^N f + L_1^N g_N + L_2^N h_N)(u,z,v) & = \kappa_1u(a_1 - 1) f'(u) + \tfrac 12 \kappa_1v(1 - 2a_1 + 2b_1)f''(u) + \varepsilon_N^f(u,z,v) \\ &  = Lf(u,v)
+ \varepsilon_N^f(u,z,v), \\ Lf(u,v) & = - \frac{\kappa_1 \kappa_3 \kappa_5 u}{\kappa_2(\kappa_4 + \kappa_5) + \kappa_3\kappa_5} f'(u) + \tfrac 12 \frac{\kappa_1 \kappa_3 \kappa_5 v}{\kappa_2(\kappa_4 + \kappa_5)+ \kappa_3\kappa_5} f''(u).
\end{align*}
So, $L$ is the generator of $U$. We also see from \eqref{eq:eq9} that \eqref{eq:T1a}--\eqref{eq:T1d} are satisfied. Finally, the compact containment condition for $(U^N)_{N=1,2,...}$ follows as in Example~\ref{ex:simHill5}.
\end{proof}

\subsection{Main example from Cappelletti und Wiuf (2016)}
\label{ss:CWmain}
\begin{figure}
    \centering
    \begin{center}
    \begin{tikzpicture}[baseline={(current bounding box.center)}, auto]
     \node[state] (E0+S0)  at (.5,0)    {$E_1+S_1$};
    \node[state] (E1) at (4,1.5)  {$E_2$};
    \node[state] (E2) at (4,-2.5)   {$E_3$};
    \node[state] (E0+S1) at (7,1.5)  {$E_1+S_2$};
      \node[state] (E0+S2) at (7,-2.5)   {$E_1+S_3$};
    \path[->]
    (E0+S0)  edge[bend left] node {$\kappa_1$} (E1)
    edge            node[swap] {$\kappa_2$} (E2)
      (E1) edge[bend left] node[swap] {$\;\;N\kappa_3$}   (E0+S0)
      edge            node {$N^3\kappa_4$} (E0+S1)
    edge[bend left] node {$N^3\kappa_5$} (E2)
    (E2) edge            node {$N^2 \kappa_6$} (E0+S2)
    edge[bend left] node {$N^2\kappa_7$} (E1)
    (E0+S2) edge[bend left=50] node {$\kappa_8$} (E0+S0);
    \end{tikzpicture}
    \end{center}
    \caption{\label{fig:CW1}Main example from \cite{CappellettiWiuf2016}.}
\end{figure}
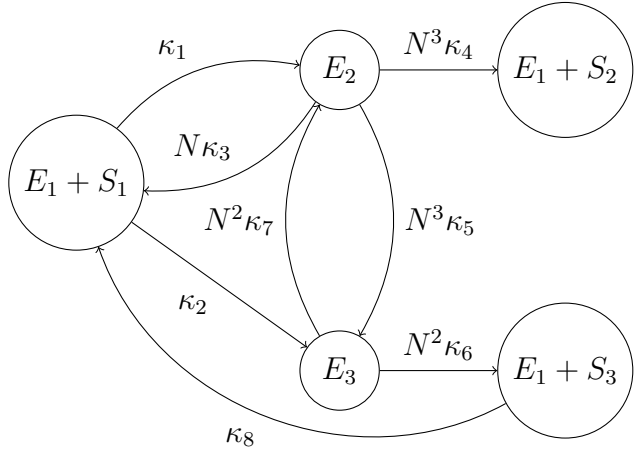

\noindent
We will now study the main example from \cite{CappellettiWiuf2016}, as given in Figure~\ref{fig:CW1}. 
Here, $E_1, E_2$ and $E_3$ are enzymes which help to transform $S_1$ into $S_2$ and $S_3$. For the abundances, we have that $E_1, E_2, E_3$ are in low abundance (order 1), and $S_1, S_2, S_3$ are in high abundance (order $N$). We denote the number of $E_1, E_2, E_3$ by $Z^N_1, Z^N_2, Z^N_3$, respectively, and the number of $S_1, S_2, S_3$ by $NV^N_1, NV^N_2, NV^N_3$, respectively. A close inspection of the reactions shows that $M := Z^N_1+Z^N_2+Z^N_3$ is a constant, and it suffices to study $Z^N = (Z_1^N, Z_2^N)$. Moreover, the reaction rates are such that $Z_1 = M$ at most times. If some molecule $E_2$ or $E_3$ is created, it reacts through either reaction $4$ or $6$, both of which occur at faster rates than the creation reactions of $E_2$ and $E_3$. For this reason, \cite{CappellettiWiuf2016} call these species \emph{intermediate}. Moreover, we also note that $V^N_{1} + V^N_{2} + V_3 + N^{-1}(Z^N_2 + Z^N_3)$ is a constant, such that it suffices to study $V^N = (V_1^N, V_3^N)$. Moreover it is easy to check that $\psi_{E_1} = \psi_{E_2} = \psi_{E_3} = 1$, so our general theory applies. 

\begin{proposition}[LLN and CLT for the system from Figure~\ref{fig:CW1}]
  Let
  \begin{equation}\label{eq:810}
    \begin{aligned}
      \lambda_1 & := \kappa_1\frac{\kappa_4(\kappa_6 + \kappa_7)}{\kappa_4(\kappa_6 + \kappa_7) + \kappa_5 \kappa_6} + \kappa_2\frac{\kappa_4  \kappa_7}{\kappa_4(\kappa_6 + \kappa_7) + \kappa_5 \kappa_6}, \\
      \lambda_2 & := \kappa_1 \frac{\kappa_5\kappa_6}{\kappa_4(\kappa_6 +  \kappa_7) + \kappa_5\kappa_6} + \kappa_2 \frac{(\kappa_4 + \kappa_5)\kappa_6}{\kappa_4(\kappa_6 +  \kappa_7) + \kappa_5\kappa_6},
    \end{aligned}
  \end{equation}
  $u=(u_1,u_3)\in\mathbb R^2$ and $v = (v_1, v_3) \in \mathbb R_+^2$, $V = (V_1, V_3)$ the solution of the ODE
  \begin{align*}
    dV_1 & = (-M(\lambda_1+\lambda_2) V_1 + M\kappa_8 V_3)dt,\\
    dV_3 & = (M\lambda_2 V_1 - M\kappa_8 V_3)dt,
  \end{align*}
  and $U = (U_1, U_3)$ solution of the SDE
  \begin{align*}
    dU_1 & = (-M(\lambda_1+\lambda_2) U_1 + M\kappa_8 U_3)dt + \sqrt{M\lambda_1 V_1} dB_1 + \sqrt{M\lambda_2 V_1} dB_2 + \sqrt{M\kappa_8 V_3} dB_3, \\ dU_3 & = (M\lambda_2 U_1 - M\kappa_8 U_3)dt - \sqrt{M\lambda_2 V_1} dB_2 + - \sqrt{M\kappa_8 V_3} dB_3
  \end{align*}
  for independent Brownian motions $B_1, B_2, B_3$. Then, letting $V^N$ be as above, setting $U^N = N^{1/2}(V^N - V)$, and if $V^N_0 \xRightarrow{N\to\infty} v$ and $U^N_0 \xRightarrow{N\to\infty} u$, we have
  $$V^N \xRightarrow{N\to\infty} V, \qquad U^N \xRightarrow{N\to\infty} U.$$
\end{proposition}

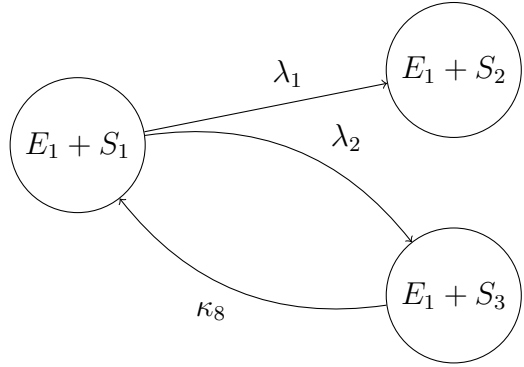
\begin{figure}     
    \begin{center}
    \begin{tikzpicture}[baseline={(current bounding box.center)}, auto]
    \node[state] (E0+S0)   at (0,0)  {$E_1+S_1$};
    \node[state] (E0+S1)  at (5,1)  {$E_1+S_2$};
    \node[state] (E0+S2)  at (5,-2) {$E_1+S_3$};
    \path[->]
     (E0+S0)  edge            node[pos=.7]{$\lambda_1$} (E0+S1)
            edge[bend left] node[pos=.6]{$\lambda_2$} (E0+S2)
     (E0+S2) edge[bend left] node{$\kappa_8$} (E0+S0);
   \end{tikzpicture}
    \end{center}
    \caption{\label{fig:CW2}Reduced system in the main example from \cite{CappellettiWiuf2016}.}
\end{figure}

\begin{remark}
  \begin{enumerate}
      \item The limiting system of the above proposition is the same as the (single-scale) system given in Figure~\ref{fig:CW2}. For the LLN, this is the same limit as obtained in \cite{CappellettiWiuf2016}. The fact that the CLT follows the same single-scale system is new.
      \item We note that our proof not only gives the limit of the system in Figure~\ref{fig:CW1}, but also for different scalings. For example, if we use $N\kappa_4$ (instead of $N^3\kappa_4$), $N\kappa_5$ (instead of $N^3\kappa_5$), $N\kappa_6$ (instead of $N^2\kappa_6$), and $N\kappa_7$ (instead of $N^2\kappa_7$), the techniques we use give a limit result as well. However, this limit is much more complex than Figure~\ref{fig:CW2}.
  \end{enumerate}
\end{remark}

\begin{proof}
Setting $\kappa'_4:=N^{2}\kappa_4$, $\kappa'_5:=N^{2}\kappa_5$, $\kappa'_6:=N\kappa_6$, $\kappa'_7:=N\kappa_7$, the resulting Markov process has the generator for $f \in \mathcal C_c^\infty(\mathbb R^2 \times \mathbb R^2)$
\begin{align*}
  G^Nf(v,z) & = N \kappa_1 v_1 z_1 (f(v-N^{-1} e_1,z-e_1+e_2) - f(v,z)) \\ & \qquad + \kappa_2 N  v_1 z_1(f(v-N^{-1}e_1,z-e_1) - f(v,z)) \\ & \qquad + N \kappa_3 z_2 (f(v + N^{-1}e_1, z + e_1 - e_2) - f(v,z)) \\ & \qquad + N \kappa'_4 z_2(f(v, z + e_1 - e_2) - f(v,z)) \\ &  \qquad + N \kappa'_5 z_2(f(v, z  - e_2) - f(v,z)) \\ & \qquad + N \kappa'_6 (M-z_1-z_2) (f(v+N^{-1}e_3, z + e_1) - f(v,z)) \\ & \qquad + N \kappa'_7  (M-z_1-z_2) (f(v, z+ e_2) - f(v,z)) \\ & \qquad + N\kappa_8 v_3 z_1(f(v + N^{-1}(e_1 - e_3), z) - f(v,z))
\\ &  =  (G_0^N + N\cdot G_2^N)f(v,z)
\end{align*}
with (writing $D_1f$ and $D_3f$ for the derivatives with respect to $v_1$ and $v_3$, and by $||D^2f||_\infty$ the supremum of all second derivatives of $f$),
\begin{equation}\label{eq:808}
    \begin{aligned}
      G_0^Nf(v,z) & = - v_1 z_1 \big(\kappa_1 D_1 f(v,z-e_1+e_2) + \kappa_2 D_1 f(v,z-e_1)\big) + \kappa_3 z_2 D_1 f(v, z + e_1 - e_2)\\ & \qquad +\kappa'_6(M-z_1-z_2)D_3 f(v,z+e_1) + \kappa_8 v_3 z_1(D_1 - D_3) f(v, z) + \varepsilon_f^N(v,z),\\ G_2^N f(v,z) & = \kappa_1 v_1 z_1 (f(v,z-e_1+e_2)-f(v,z)) +\kappa_2 v_1 z_1 (f(v,z-e_1)-f(v,z))\\
      & \qquad +\kappa_3 z_2 (f(v,z+e_1-e_2)-f(v,z))+\kappa'_4 z_2 (f(v,z+e_1-e_2)-f(v,z))\\
      & \qquad +\kappa'_5 z_2(f(v,z-e_2)-f(v,z))+\kappa'_6 (M-z_1-z_2) (f(v,z+e_1)-f(v,z))\\
      & \qquad +\kappa'_7  (M-z_1-z_2) (f(v,z+e_2)-f(v,z)),\\
      |\varepsilon_f^N(v,z)| &\leq \tfrac 12   N^{-1} (v_1 z_1 (\kappa_1 + \kappa_2) + \kappa_3 z_2 + \kappa'_6(M-z_1-z_2) + \kappa_8 v_3 z_1)||D^2f||_\infty.
    \end{aligned}
\end{equation}
For bounding the two fast variables, we have that $0\leq Z_1^N, Z_2^N \leq M$, and
\begin{align*}
    \int_0^T (Z_N^1(t))^k dt = O(1), \quad     \int_0^T (Z_N^2(t))^k dt = O(N^{-2}), \quad     \int_0^T (M-(Z_N^1+Z_N^2)(t))^k dt = O(N^{-1}).
\end{align*}
1. The compact containment condition for $(V^N)_{N=1,2,...}$ follows from the fact that $V_1^N + V_2^N + V_3^N + N^{-1}(Z_2^N + Z_3^N)$ is a constant. Moreover, \eqref{eq:T1a}--\eqref{eq:T1d} will follow from the above boundedness assertions. So, for the generator calculations, assume that $f \in \mathcal C_c^\infty(\mathbb R^2)$ only depends on $v$. Taking, for some functions $a_1, ..., a_4$ (recall this ansatz from \eqref{eq:g1})
\begin{align*}
g_N(v,z) & = (a_1z_1 + a_2z_2)D_1 f(v) + (a_3z_1 + a_4z_2)D_3 f(v),    
\end{align*}
we find
\begin{align*}
  ( & G_0^N f + G_2^N g_N)(v,z) - \varepsilon_f^N(v,z)\\ & =  - v_1 z_1 \big(\kappa_1 D_1 f(v) +\kappa_2 D_1 f(v)\big) + \kappa_3 z_2 D_1 f(v) +\kappa'_6(M-z_1-z_2)D_3 f(v)\\ & \qquad+ \kappa_8 v_3 z_1(D_1 - D_3) f(v)+\kappa_1 v_1 z_1((a_2-a_1)D_1 f(v)+(a_4-a_3)D_3 f(v))\\ & \qquad+\kappa_2v_1 z_1(-a_1D_1 f(v)-a_3D_3 f(v)) +\kappa_3 z_2 ((a_1-a_2)D_1 f(v)+(a_3-a_4)D_3 f(v))\\ & \qquad+\kappa'_4 z_2((a_1-a_2)D_1 f(v)+(a_3-a_4)D_3 f(v))+\kappa'_5 z_2(-a_2D_1 f(v)-a_4D_3 f(v))\\ & \qquad+\kappa'_6 (M-z_1-z_2)(a_1D_1 f(v)+a_3 D_3 f(v))+\kappa'_7  (M-z_1-z_2)(a_2D_1 f(v)+a_4D_3 f(v))\\
  &=M((\kappa'_6 a_1+\kappa'_7 a_2)D_1 f(v)+(\kappa'_6(a_3+1)+\kappa'_7a_4) D_3 f(v)) \\ &\qquad +z_1\Big(D_1 f(v)\big(-v_1 \kappa_1 -v_1\kappa_2 +\kappa_8 v_3 +\kappa_2v_1(a_2-a_1)-\kappa_2 v_1 a_1 -\kappa'_6 a_1-\kappa'_7 a_2\big)\\ &\qquad +D_3 f(v)\big(-\kappa'_6 -\kappa_8v_3  +\kappa_1v_1(a_4-a_3)-\kappa_2 v_1 a_3 -\kappa'_6 a_3-\kappa'_7 a_4\big)\Big)\\ &\qquad +z_2\Big(D_1 f(v)\big( \kappa_3+\kappa_3 (a_1-a_2)+\kappa'_4(a_1-a_2)-\kappa'_5 a_2-\kappa'_6 a_1-\kappa'_7 a_2\big)\\ &\qquad +D_3 f(v)\big(-\kappa'_6 +\kappa_3 (a_3-a_4)+\kappa'_4(a_3-a_4)-\kappa'_5 a_4-\kappa'_6 a_3-\kappa'_7 a_4\big)\Big).
\end{align*}
If we set $a_1,...,a_4$ such that the terms proportional to $z_1D_1, z_1D_3,z_3D_1$ and $z_3D_3$ vanish, we obtain
\begin{align*}
Gf(v) & := \lim_{N\to\infty} (G_0^Nf + G_1^Ng)(v,z) - \varepsilon_f^N(v,z) \\ & =  M\lim_{N\to\infty} (a_1\kappa'_6+a_2\kappa'_7)D_1 f(v)+M(\kappa'_6(a_3+1)+\kappa'_7)D_3 f(v) \\ & = M(-\lambda_1 v_1D_1f + \lambda_2 v_1(D_3f - D_1f) + \kappa_8 v_3(D_1f - D_3f))
\end{align*}
for $\lambda_1, \lambda_2$ as in \eqref{eq:810}.
Since $G$ is the generator of $V$, we are done.

~

\noindent
2. The generator of $(U^N, Z^N)$ reads
\begin{align*}
    L^Nf(u,z) & = 
    \kappa_1 (Nv_1 + N^{1/2} u_1) z_1 (f(u-N^{-1/2} e_1,z-e_1+e_2) - f(u,z)) \\ & \qquad + \kappa_2 (N  v_1 + N^{1/2}u_1) z_1(f(u-N^{-1/2}e_1,z-e_1) - f(u,z)) \\ & \qquad + N \kappa_3 z_2 (f(u + N^{-1/2}e_1, z + e_1 - e_2) - f(u,z)) \\ & \qquad + N \kappa'_4 z_2(f(u, z + e_1 - e_2) - f(u,z)) \\ &  \qquad + N \kappa'_5 z_2(f(u, z  - e_2) - f(u,z)) \\ & \qquad + N \kappa'_6 (M-z_1-z_2) (f(u+N^{-1/2}e_3, z + e_1) - f(u,z)) \\ & \qquad + N \kappa'_7  (M-z_1-z_2) (f(u, z+ e_2) - f(u,z)) \\ & \qquad + \kappa_8 (Nv_3 + N^{1/2}u_3) z_1(f(u + N^{-1/2}(e_1 - e_3), z) - f(u,z)) \\ & \qquad - MN^{1/2}((-(\lambda_1+\lambda_2) v_1 + \kappa_8 v_3)D_1f(u) + (\lambda_2 v_1 - \kappa_8 v_3)D_3f(u)) \\ & = (L_0^Nf+N^{1/2}L_1^Nf+NL_2^Nf)(u,z)
\end{align*} 
with (writing $D_{11},D_{13}$ and $D_{33}$ for all the second derivatives with respect to $v_1$ and $v_3$, and $||D^3f||_\infty$ for the supremum of all third derivatives of f)
\begin{align*}
  L_0^Nf(u,z)&=-\kappa_1u_1z_1 D_1f(u,z-e_1+e_2)+\tfrac 12\kappa_1v_1z_1 D_{11}f(u,z-e_1+e_2) -\kappa_2u_1z_1 D_1f(u,z-e_1) \\  &\qquad +\tfrac 12\kappa_2 v_1z_1 D_{11}f(u,z-e_1)+\tfrac 12\kappa_3 z_2 D_{11}f(u, z + e_1 - e_2) \\ & \qquad + \tfrac 12 \kappa_6'(M-z_1-z_2) D_{33}f(u,z+e_1) + \kappa_8 u_3z_1(D_1 - D_3)f(u,z) \\ & \qquad + \tfrac 12 \kappa_8 v_3 z_1(D_{11} - 2D_{13} + D_{33})f(u,z) +  \varepsilon_f^N(u,z,v),\\
  L_1^Nf(u,z)&=-\kappa_1 v_1z_1 D_1f(u,z-e_1+e_2)+\kappa_1u_1z_1(f(u,z-e_1+e_2)-f(u,z)) \\   &\qquad - \kappa_2v_1z_1D_1f(u,z-e_1) +\kappa_2u_1z_1(f(u,z-e_1)-f(u,z))+\kappa_3 z_2D_1f(u,z+e_1-e_2) \\ & \qquad + \kappa_6'(M-z_1-z_2)D_3f(u,z+e_1)  + \kappa_8v_3z_1(D_1-D_3)f(u,z) \\ & \qquad- M((-(\lambda_1+\lambda_2) v_1 + \kappa_8 v_3)D_1f(u,z) + (\lambda_2 v_1 - \kappa_8 v_3)D_3f(u,z)),\\
  L_2^Nf(u,z)&=\kappa_1v_1z_1(f(u,z-e_1+e_2)-f(u,z))+\kappa_2v_1z_1(f(u,z-e_1)-f(u,z))\\  &\qquad +\kappa_3z_2(f(u,z+e_1-e_2)-f(u,z))+\kappa'_4 z_2 (f(u,z+e_1-e_2)-f(u,z))\\
  &\qquad +\kappa'_5 z_2(f(u,z-e_2)-f(u,z)) + \kappa_6'(M-z_1-z_2)(f(u,z+e_1) - f(u,z)) \\ & \qquad + \kappa'_7  (M-z_1-z_2) (f(u,z+e_2)-f(u,z)), \\ |\varepsilon_f^N(u,z,v)| & \leq N^{-1/2} (((\kappa_1 + \kappa_2) u_1 z_1 + \kappa_8 u_3 z_1)||D^2f||_\infty \\ & \qquad \qquad \qquad + \tfrac 16 (v_1 z_1 (\kappa_1 + \kappa_2) + \kappa_3 z_2 + \kappa'_6(M-z_1-z_2) + \kappa_8 v_3 z_1)||D^3f||_\infty).
\end{align*}
With $g_N(u,z) = (a_1z_1 + a_2z_2)D_1f(u) + (a_3z_1 + a_4z_2)D_3f(u)$, this time depending on $u$ instead of $v$ we obtain, if $f$ only depends on $u$, that $L_1^Nf+L_2^Ng=o(N^{-1/2})$. We choose the ansatz (compare with \eqref{eq:hNgen})
\begin{align*}
  h_N(u,z) & =(b_1z_1+b_2z_2)D_1f(u) + (b_3z_1+b_4z_2)D_3f(u) \\ & \qquad + \Big(c_1 z_1 + c_2 z_2 + \binom{z_1}{2} d_1 + \binom{z_2}{2} d_2 + \binom{z_1+z_2}{2} d_3 \Big) D_{11} f(u)
  \\ & \qquad + \Big(c_3 z_1 + c_4 z_2 + \binom{z_1}{2} d_4 + \binom{z_2}{2} d_5 + \binom{z_1+z_2}{2} d_6 \Big) D_{33} f(u)
  \\ & \qquad + \Big(c_5 z_1 + c_6 z_2 + \binom{z_1}{2} d_7 + \binom{z_2}{2} d_8 + \binom{z_1+z_2}{2} d_9 \Big) D_{13} f(u)
\end{align*}
for some $b_1,...,b_4, c_1,...,c_{6}, d_1,...,d_9$. As in the above examples, plugging $f$, $g_N$ with $a_1,...,a_4$ as above, and $h_N$ in $L_0f + L_1g_N + L_2h_N$ leads to a term which depends only on $u$ for the correct choice of $b_1,...,b_4, c_1,...,c_{6}, d_1,...,d_9$. The corresponding linear system can readily be solved, and the result is (leaving all computations to a computer algebra system such as sagemath) leads for $N\to\infty$ to the limit
\begin{align*}
  Lf(u) & = \lim_{N\to\infty} (L_0^Nf+L_1^Ng+L_2^Nh - \varepsilon_f^N)(u) \\ & = M(-(\lambda_1 + \lambda_2)u_1 + \kappa_8 u_3) D_1f(u) + M(\lambda_2 u_1 - \kappa_8 u_3) D_3f(u) \\ & \qquad + \tfrac 12 (\lambda_1v_1 D_{11}f(u) + (\lambda_2 v_1 + \kappa_8 v_3) (D_{11} - 2D_{13} + D_{33})f(u).
\end{align*}
This is the generator of $(U_1,U_2)$. 
Again, showing the compact containment condition of $(U^N)_{N=1,2,...}$ as well as \eqref{eq:T1a}--\eqref{eq:T1d} works with the same arguments we already saw in the first examples.
\end{proof}

\begin{appendix}
\section{An immigration death process}
\label{app:A}
We will now give a tool which helps to derive the conditions \eqref{eq:T1a}--\eqref{eq:T1d} in concrete examples. 
\begin{lemma}[A process with immigration and death]
\label{l:ccc}Let $\beta \geq \alpha \geq 0, C_0\geq 0, C_1 > 0$, and $k, T, a >0$, and for each $N\in\mathbb N$, let $X^N = (X^N_t)_{t\geq 0}$ be a Markov-jump-process which increases at time $t$ by $1$ at rate $C_0^N N^\alpha$ and decreases at time $t$ by $1$ at rate $C_1^N N^\beta X^N(t)$ and $X^N_0=0$, where $0\leq C_0^N \leq C_0$ and $C_1^N \geq C_1$. Then, 
\begin{align}
    \label{eq:ccc1}
     \sup_{N\in\mathbb N}\mathbb{E}\Big[\int_0^T N^{\beta-\alpha}(X^N_t)^kdt\Big]<\infty
\end{align}
and 
\begin{align}
    \label{eq:ccc2}
    N^{-a}\mathbb E[\sup_{0\leq t\leq T} (X^N_t)^k] \xrightarrow{N\to\infty} 0.
\end{align}
\end{lemma}


\begin{proof}
We use some constant $C$, which does not depend on $N$, but which may change from line to line. With Fubini, we see that \eqref{eq:ccc1} follows from 
\begin{align*}
    \mathbb E[(X_t^N)^k] \leq C N^{\alpha-\beta}, \qquad t\geq 0.
\end{align*}
We start the proof of this assertion by calculating 
\begin{align*}
    \frac{d}{dt} \mathbb E[(X_t^N)^k] = C_0^N N^\alpha \mathbb E[(X_t^N+1)^{k} - (X_t^N)^{k}] + C_1^N N^\beta \mathbb E[X_t^N(X_t^N-1)^k - (X_t^N)^{k+1}].
\end{align*}
Using induction, we start with $k=1$ and can solve this differential equation with $\mathbb E[X^N_0] = 0$, i.e.
$$\mathbb{E}[X_t^N]=\frac{C_0^N}{C_1^N}N^{\alpha-\beta}(1 - e^{-C_1^N N^{\beta}t}) \leq C N^{\alpha-\beta}.
$$
From here, if we have shown the case $k-1$, we see that $\mathbb E[(X_t^N+1)^{k} - (X_t^N)^{k}] \leq C$, and the differential equation yields
\begin{align*}
    \frac{d}{dt} \mathbb E[(X_t^N)^k] \leq C N^\alpha + C_1 N^\beta \mathbb E[(X_t^N)^k]. 
\end{align*}
From here, Gronwall's Lemma shows the assertion for $k$.

For \eqref{eq:ccc2}, we only need to consider the case $\beta = \alpha$ since this process dominates the process with $\beta > \alpha$. Note that $X^N$ can be seen as the size of a population undergoing immigration (at individual rate $C_0 N^\alpha$) and death (at individual rate $C_1 N^\beta = C_1 N^\alpha$). Let $\tau_n := \inf\{t: X^N_t > n\}$ and consider $X^N_{\tau_n}$. At this time, the population consists of individuals immigrating at times before $\tau_n$, and we consider the oldest such individual. Necessarily, this individual has seen $n$ individuals immigrating before $\tau_n$ which has probability $\gamma_N^n$ for $\gamma_N := C_0^N / (C_0^N + C_1^N) \leq C_0 / (C_0 + C_1)=:\gamma$. The rate by which individuals immigrate which see $n$ immigration events before their own death is $C_0^N N^\alpha \gamma^n_N$, i.e.
\begin{align*}
    \mathbb P(\sup_{0\leq t\leq T} X^N_t > n) & \leq 1 - \exp\Big( - C_1 N^\alpha \gamma^n_N T\Big).
\end{align*}
We therefore write
\begin{align*}
    \mathbb E[\sup_{0\leq t\leq T} (X^N_t)^k] & = k \int_0^\infty x^{k-1} \mathbb P(\sup_{0\leq t\leq T} X^N(t) > x) dx
    \\ & \leq k\int_0^{(\log N)^2} x^{k-1} dx + 2k C_1 N^\alpha\int_{(\log N)^2}^\infty x^{k-1} \gamma^x dx \leq C(\log N)^{2k},
\end{align*}
and the assertion follows.
\end{proof}

\end{appendix}

\subsubsection*{Acknowledgements}
We thank Martin Hutzenthaler for valuable comments on possible improvements of Theorem~\ref{T1}.


\end{document}